\documentclass[12pt]{amsart}

\usepackage{amssymb}
\usepackage{pb-diagram}
\usepackage{enumitem} 
\usepackage[dvipsnames]{xcolor}
\usepackage{tikz-cd}
\usepackage{graphicx}
\usepackage{amsmath}
\numberwithin{equation}{section}
\usepackage[all]{xy}

\newtheorem{Theorem}{Theorem}[section]
\newtheorem{Proposition}[Theorem]{Proposition}
\newtheorem{Lemma}[Theorem]{Lemma}
\newtheorem{Corollary}[Theorem]{Corollary}
\newtheorem{Definition}[Theorem]{Definition}
\newtheorem{Remark}[Theorem]{Remark}

\numberwithin{equation}{section}

\begin{document}

\baselineskip=16pt

\title{On the moduli space of coherent systems of type $(2,c_1,c_2,2)$ on projective plane.}

\author{O. Mata-Guti\'errez}

\address{Departamento de Matem\'aticas\newline
Centro Universitario de Ciencias Exactas e Ingenier\'ias \newline
Universidad de Guadalajara\\
\newline  Avenida Revoluci\'on 1500\\ Guadalajara, Jalisco, M\'exico.}
\email{osbaldo.mata@academicos.udg.mx}

\author{L. Roa-Leguizam\'on}

\address{Universidad de Los Andes, Departamento de matem\'aticas \newline Carrera 1 \#18A-12, 111 711   \newline Bogota, Colombia.}
\email{leonardo.roa@cimat.mx}

\author{H. Torres-L\'opez}

\address{CONACyT - U. A. Matem\'aticas, U. Aut\'onoma de
Zacatecas
\newline  Calzada Solidaridad entronque Paseo a la
Bufa, \newline C.P. 98000, Zacatecas, Zac. M\'exico.}

\email{hugo@cimat.mx}

\thanks{}

\subjclass[2010]{}


\date{\today}

\maketitle

\begin{abstract}
We study the moduli space of coherent systems in $\mathbb{P}^2$ using the Segre invariant. We obtain necessary conditions for the existence of $\alpha$-semistable coherent systems $(E,V)$ of type $(2,c_1,c_2, k),$ with $k \geq 2$. Afterwards, we give numerical conditions to the nonemptiness of the moduli space
and compute the critical values depending of the Chern
classes. Finally, we give some topological properties of the flips.
\end{abstract}

\section{Introduction}

Let $X$ be a smooth and irreducible algebraic surface. A coherent system on $X$ of type
$(n,c_1,c_2, k)$ is a pair $(\mathcal{E},V)$ where
$\mathcal{E}$ is a coherent sheaf of rank $n$ on $X$
with $c_i= c_i(\mathcal{E})\in H^{2i}(X,\mathbb{Z})$, $i=1,2$ and
$V\subseteq H^{0}(X,\mathcal{E})$ a linear subspace of
dimension $k$. In \cite{Potier}, Le Potier introduces the notion of coherent systems on a projective variety
and define  semistability. 
Subsequently,  by fixing a polynomial $\alpha$ and the
type $(n,c_1,c_2, k)$, Le Potier constructs the 
moduli space $\mathbf{S}_{\alpha}(n,c_1,c_2, k)$
of $\alpha$-semistable coherent systems, which is a
projective variety. Later, Min He provided a
smoothness criterion, studied critical values, and 
determined  topological and geometric properties 
of the moduli space $\mathbf{S}_{\alpha}(n,c_1,c_2, k)$
(see \cite{Min}).  
Both of these constructions were inspired by Simpson's
construction of the moduli space of coherent sheaves \cite{Simpson}.

The geometry of the moduli space
$\mathbf{S}_{\alpha}(n, c_1 , c_2, k)$ depends on $\alpha$.
However, it is known that there exists a finite number of values 
$\alpha_{1},\alpha_{2},\ldots , \alpha_{r}$ 
called  critical values such that if $\alpha$ and $\alpha'$
are in the interval $(\alpha_{i}, \alpha_{i+1})$,
then $\mathbf{S}_{\alpha}(n,c_1,c_2, k)= \mathbf{S}_{\alpha'}(n,c_1,c_2, k)$.  On the other hand, if $\alpha' \leq \alpha_{i} < \alpha$,
then $\mathbf{S}_{\alpha}(n,c_1,c_2, k)$ and $\mathbf{S}_{\alpha'}(n,c_1,c_2, k)$
are birational. The difference between consecutive moduli spaces $\mathbf{S}_{\alpha}(2,c_1,c_2,1)$ was studied by Min He (see \cite{Min}); however, little is known about the moduli space of $\alpha$-stable coherent systems of type 
$\mathbf{S}_{\alpha}(2,c_1,c_2,k)$  with $k> 1$.

In \cite{Min}, Min He studied topological and geometric properties of the moduli space of
coherent systems on the projective plane of type $(2,c_1,c_2, 1)$
and showed that the moduli space $\mathbf{S}_{\alpha}(2,c_1,c_2,1)$ is projective and irreducible.For this, 
 Min He used the Bertini's  scheme as an essential tool.
He determined the subschemes $\Sigma_{+}$ and $\Sigma_{-}$ of the moduli space
$\mathbf{S}_{\alpha_{i}}(n,c_1,c_2, k)$ where $\Sigma_{+}$
consists of points in $\mathbf{S}_{\alpha_{i}}(n,c_1,c_2, k)$
which are not  $\alpha$-stable if $\alpha< \alpha_i$ and 
$\Sigma_{-}$ consists of points in 
$\mathbf{S}_{\alpha_{i-1}}(n,c_1,c_2, k)$
that are not stable $\alpha$ if $\alpha> \alpha_i$.

When $X$ is a complex smooth and irreducible algebraic curve, the moduli space
$\mathbf{S}_{\alpha}(n,c_1,k)$ has been extensively studied
by several authors (see, for instance, \cite{BGMMN}, 
\cite{Bradlow},\cite{Brambila},\cite{New}).  
Much is known about topological and geometric properties
of the moduli space $\mathbf{S}_{\alpha}(n,c_1,k)$ when $k \leq n$. 
For the case $k > n$ little is known and there are several open
questions related to the structure of the moduli space (see \cite{New} for more details).

The aim of this paper is to study the moduli space of
coherent systems of type $(2,c_1,c_2,2)$ on the
projective plane $\mathbb{P}^2$. 
We will focus on coherent systems 
$(\mathcal{E},V)$  where $\mathcal{E}:=E$ is
a vector bundle.  According to the properties of GIT, 
it is known that the conditions for being locally free
and stable are open. 
Therefore,  we consider the open subscheme
$\mathbf{M}_{\alpha}(n,c_1,c_2, k)\subseteq \mathbf{S}_{\alpha}(n,c_1,c_2,k)$
which denotes the moduli space of $\alpha$-stable 
coherent systems $(E,V)$
of type $(n,c_1,c_2, k)$ on $\mathbb{P}^{2}$.   
Similarly, we denote by
$\widetilde{\mathbf{M}}_{X,\alpha}(n,c_1, c_2, k) \subset \mathbf{S}_{X,\alpha}(n,c_1, c_2, k)$
the open subscheme of $\alpha$-semistable coherent systems $(E,V)$.

Restricting us to this kind of coherent systems we will
use Cayley-Bacharach property 
(see Theorem \ref{Cayley-Bacharach}) 
and the Segre invariant for rank two vector bundles on
$\mathbb{P}^2$ to get topological and geometric properties
of the open subscheme $\mathbf{M}_{\alpha}(2,c_1,c_2,2) \subset \mathbf{S}_{\alpha}(2,c_1,c_2,2)$. We also describe the different
critical values and the geometry of flips.  

According to Theorem \ref{invariantedesegre} and Theorem \ref{invariantedesegrenegativo}  any rank two vector bundle
$E$ on $\mathbb{P}^2,$ with Chern classes $c_1$ and $c_2$, 
can be fit in an exact sequence
\begin{equation}\label{ext}
    0\rightarrow \mathcal{O}_{\mathbb{P}^2}(r-s)\rightarrow E\rightarrow \mathcal{O}_{\mathbb{P}^2}(r+s-t)\otimes I_Z\rightarrow 0,
\end{equation} 
where $\mathcal{O}_{\mathbb{P}^2}(r-s)$ is the maximal subbundle of $E$.
By fixing a polynomial $\alpha(m)=am+b\in \mathbb{Q}[m]$, 
we will use the extension (\ref{ext}) to study coherent systems
$(E,V)$ of type $(2,c_1,c_2,2)$ and describe the
$\alpha$-stability of $(E,V)$ in terms of its maximal
coherent subsystem (see Definition \ref{Def-maximalsub} and Corollary \ref{estabilidadconmaximal}). Also, we obtain necessary conditions for the existence of $\alpha$-stable coherent systems $(E,V)$ of type $(2,c_1,c_2, k),$ with $k\geq 2.$ The result is presented as follows.

\begin{Theorem}(Theorem \ref{nonemptinesspar})
 Let $\alpha(m)=am+b\in \mathbb{Q}[m]$ with $a\in \mathbb{Q}_{+}$ and let $c_1=2r-t>0$, $t \in \{0,1\}$. 
Then  $\mathbf{M}_{\alpha}(2,2r-t,c_2,2)\neq \emptyset$ if one of the following conditions is satisfied:
\begin{enumerate}
    \item If $c_2\ge r^2-t+2.$  
    \item If $r^2-rt\leq c_2 <r^2-t+2 ,$ $r\geq 2+t$ and $a>\frac{t}{2}.$
    \item  If $s_0<0$ is an integer such that $t+2-r\leq s_0,$ $\frac{t}{2}-s_0< a$ and $r^2-s_0^2+(s_0-r)t\leq c_2.$ 

\item If $c_2=2r-t, 2r-t-1$ and $a>r-1-\frac{t}{2}$.

\end{enumerate}
\end{Theorem}

We provide an explicit description of the critical values in Theorem \ref{criticalThm1}, Theorem \ref{criticalThm2} and Theorem \ref{criticalThm3} whenever the first Chern class is even. In Theorem \ref{criticalThmodd} we show the description when the first Chern class is odd.  Finally, we describe the differences between the consecutive moduli space $\mathbf{M}_\alpha(2,c_1,c_2, 2)$
by studying the sets $\Sigma_+$ and $\Sigma_-$, and we obtain the following result;

\begin{Theorem}(Theorem \ref{ThmFlip})
There exists a vector bundle $W^{-}_t$ over $G\times H_t$  
such that $\Sigma^{-}_{i,t}=\mathbb{P}(W^{-}_t)$. 
In particular, $\Sigma^{-}_{i,t}$  is irreducible of dimension 
\begin{eqnarray*}
3c_2-3r^2+s^2+3s+2h^0(\mathbb{P}^2,\mathcal{O}_{\mathbb{P}^2}(r-s))-6+\frac{t}{2}(6r-2s-3-t),
\end{eqnarray*}
where \[
\Sigma^{-}_{i,t} = \{ (E,V)\in  \mathbf{M}_{i}(2,2r-t,c_2,2)  ~\vert~ \text{$(E,V)$ is not $\alpha$-stable for $\alpha>\alpha_{i+1}$} \}. 
\]
\end{Theorem}

The paper is organized as follows; in Section \ref{sect-VBSI} we introduce the results on the Segre invariant for rank two vector bundles on $\mathbb{P}^2$. In Section  \ref{Sec-CS} we show the fundamental outcomes concerning to coherent systems, specifically those related to coherent systems on $\mathbb{P}^2$.
Section \ref{sect-SICS} we develop the connection between the Segre invariant and $\alpha$-stability. In Section \ref{sect-Nonempt} we determine the conditions for the non-emptiness of the moduli space $\mathbf{M}_{\alpha}(2,c_1,c_2, 2)$.  Section \ref{sect- CV} provides a description of critical points and Section \ref{sect-Flips} describes topological properties of the flips.

\section{Vector bundles on $\mathbb{P}^{2}$ and Segre invariant.}\label{sect-VBSI}
Let $E$ be a vector bundle on $\mathbb{P}^2$ and $\mathcal{O}_{\mathbb{P}^2}(1)$ be an ample line bundle on $\mathbb{P}^2$.
The slope of $E$ is defined as 
$$
\mu(E)=\frac{c_1(E)}{\text{rk}(E)},
$$
where $c_1(E)$ and $\text{rk}(E)$ denote the first Chern class and the rank of $E$, respectively.
We say that the vector bundle $E$ is $\mu$-stable (resp. $\mu$-semistable)
if for any proper subbundle $F\subset E$ we have
$$
\mu(F) < \mu(E) \,\,\,\, (\text{resp.} \leq).\\
$$

As is well known, a vector bundle $E$ is of type $(n,c_1,c_2)$ if $rk(E)=n$, $c_1(E)=c_1$, and $c_2(E)=c_2$. The moduli space of $\mu$-stable vector bundles of type $(n,c_1,c_2)$
on a non-singular projective surface was constructed by Maruyama in \cite{Ma}. We denote by $\mathcal{M}_{\mathbb{P}^2}(n,c_1,c_2)$ or simply by $\mathcal{M}(n,c_1,c_2)$, the moduli space of $\mu$-stable vector bundles on $\mathbb{P}^2.$

In this paper, we will use basically three tools for studying vector bundles and coherent systems: the Cayley-Bacharach property, the Segre invariant, and short exact sequences of sheaves.  We recall the main results of these topics;

\begin{Theorem}\cite[Theorem 5.1.]{Huybrechts-Lehn} \label{Cayley-Bacharach}
Let $Z \subset X$ be a local complete intersection of codimension two in the non-singular projective surface $X$,
and let $L$ and $L_0$ be two line bundles on $X$.
Then, there exists an extension
\[0 \longrightarrow L \longrightarrow E \longrightarrow L_0 \otimes I_Z \longrightarrow 0\]
such that $E$ is a vector bundle if and only if the pair $(L^{-1}
\otimes L_0 \otimes \omega_X,Z)$ satisfies the Cayley-Bacharach
property:
\begin{align*}
 (CB) \,\,\,  & \text{if $\widetilde{Z} \subset Z$ is a sub-scheme with
 $\ell({\widetilde Z})= \ell(Z)-1$ and }\\
 & \text{$s \in H^0(L^{-1} \otimes L_0 \otimes \omega_X)$  with  
 $s\vert _{\widetilde{Z}}=0$, then $s\vert_Z=0$}.
\end{align*}
\end{Theorem}

This result is provided for non-singular projective surfaces. However, we will apply it to the projective plane $X=\mathbb{P}^2$. The Segre invariant for vector bundles on $\mathbb{P}^2$ is defined as follows; Let $E$ be a vector bundle of rank two on $\mathbb{P}^2$.
The Segre invariant of $E,$ denoted by $S(E),$ is the integer number given by
$$
S(E):= c_1(E)-2\max_{L\subset E}\{c_1(L)\}
$$
where the maximum is taken over all line subbundles $L$ of $E$ (see \cite{Leo}).
If $L \subset E$ is a line subbundle such that $S(E)=c_1(E)-2c_1(L),$
then $L$ is called the maximal line subbundle of $E$ and is denoted by $L_{max}$.

It is easy to check that $E$ is a stable (resp. semistable) rank two vector bundle, if and only if $S(E) > 0 $ (resp. $\geq 0$) 
and $S(E)< 0$ if and only if
$E$ is $\mu$-unstable.  Here, we note that the stability of the vector bundle $E$ is determined by its Segre invariant; moreover, it is determined by its maximal line bundle.  The term invariant is used because it does not change under the tensor product of line bundles, i.e. $S(E)= S(E \otimes L)$.

In \cite{Ma3} and \cite{RTZ} it is shown that the degree of line bundles
of $E$ is bounded from above. Therefore, for every vector bundle $E$, there exists a unique maximal line subbundle.
The maximal subbundle allows us to study geometric and topological properties of the moduli space of vector bundles via Segre invariant and extensions, this is shown in following theorems;

The following result was given in \cite[Theorem 3.1 and  Corollary 3.5]{RTZ}. We reformulate the theorem and the proof to emphasize the properties that will be used in the subsequent sections.

\begin{Theorem}\label{invariantedesegre}
    Let $r,c_2$ be integer numbers and $t\in\{0,1\}$ with $c_2\geq r^2+2-t$ and $s\in \mathbb{N}$. Then the following statements are equivalent;
    
    \begin{enumerate}
        \item There exists a vector bundle $E\in \mathcal{M}_{\mathbb{P}^2}(2,2r-t,c_2)$
    with $S(E)=2s-t.$ 
        \item  $c_2\geq s^2+s+r^2-t(r+s).$
        \item There exists a zero cycle $Z$ which is a locally complete intersection of length $\ell(Z) = c_2 + s^2-r^2+t(r-s)$ and it is not contained in any curve of degree $2s-1-t$.
            Moreover, $Z$ fits into the following exact sequence
    \begin{eqnarray*}
     0\rightarrow \mathcal{O}_{\mathbb{P}^2}(r-s)\rightarrow E \rightarrow \mathcal{O}_{\mathbb{P}^2}(r+s-t) \otimes I_Z \rightarrow 0,
    \end{eqnarray*}
where $E$ vector bundle.
    \end{enumerate}

\end{Theorem}
\begin{proof}
The proof of $(1)$ if and only if $(2)$ follows from \cite[Theorem 3.1 and  Corollary 3.5]{RTZ}. 

$(2)\Rightarrow (3):$ Let $Z$ be a zero cycle which is a locally complete intersection of length $\ell(Z) = c_2 + s^2-r^2+t(r-s)$ in general position. The condition $c_2\geq s^2+s+r^2-t(r+s)$ implies that $\ell(Z)\geq h^0(\mathbb{P}^2,\mathcal{O}_{\mathbb{P}^2}(2s-1-t))$. Therefore $h^0(\mathbb{P}^2,\mathcal{O}_{\mathbb{P}^2}(2s-1-t)\otimes I_Z)=0$. This show the existence of $Z$, now we will show the existence of the exact sequence.

 Let ${\tilde Z}\subset Z$ be a zero cycle such that $\ell({\tilde Z})=\ell(Z)-1$.    The pair $(\mathcal{O}_{\mathbb{P}^2}(2s-3-t),Z)$ satisfies the Cayley-Bacharach property (see Theorem \ref{Cayley-Bacharach}), then there exists an extension 
\begin{equation}\label{38imaximal}
  0 \longrightarrow \mathcal{O}_{\mathbb{P}^2}(r-s) \longrightarrow E \longrightarrow \mathcal{O}_{\mathbb{P}^2}(r+s-t) \otimes I_Z \longrightarrow 0. 
\end{equation} 
where $E$ is locally free and has Chern classes $c_1(E)=2r-t$  and $c_2(E) = c_2$.

 $(3) \Rightarrow (1):$ 
 
 Let $Z\subset \mathbb{P}^2$ be a local complete intersection of codimension 2 with length $\ell(Z)=  c_2 + s^2-r^2+t(r-s)$  such that $Z$ is not contained in any curve of degree $2s-1-t$,  and let 

 \begin{equation}\label{38maximal}
  0 \longrightarrow \mathcal{O}_{\mathbb{P}^2}(r-s) \longrightarrow E \longrightarrow \mathcal{O}_{\mathbb{P}^2}(r+s-t) \otimes I_Z \longrightarrow 0, 
\end{equation} 

the extension (\ref{38maximal}) that fits $Z,$ hence $E$ is a vector bundle such that $c_1(E)=2r-t$  and $c_2(E) = c_2$.
Since $Z$ is not contained in any curve of degree  $2s-1-t$ follows that $H^0(\mathbb{P}^2, \mathcal{O}_{\mathbb{P}^2}(s-t) \otimes I_Z) = H^0(\mathbb{P}^2,E(-r))=0$. Since $c_1(E(-r))=-t$, it follows that  the vector bundle $E(-r)$ and $E$ is stable.

Finally, we prove that  $\mathcal{O}_{\mathbb{P}^2}(r-s)$ is maximal.  Let $\mathcal{O}_{\mathbb{P}^2}(l)$ be a line 
bundle with $l<s-r$, from the exact sequence (\ref{38maximal}) we have 
\[0 \longrightarrow \mathcal{O}_{\mathbb{P}^2}(r-s+l) \longrightarrow E(l) \longrightarrow \mathcal{O}_{\mathbb{P}^2}(r+s+l-t) \otimes I_Z \longrightarrow 0\]

which induces the long exact sequence in cohomology
\begin{align*}
    0 \longrightarrow & H^0(\mathbb{P}^2, \mathcal{O}_{\mathbb{P}^2}(r-s+l)) \longrightarrow H^0(\mathbb{P}^2, E(l)) \longrightarrow H^0(\mathbb{P}^2, \mathcal{O}_{\mathbb{P}^2}(r+s+l-t)\otimes I_Z) \longrightarrow \\
    & H^1(\mathbb{P}^2, \mathcal{O}_{\mathbb{P}^2}(r-s+l)) \longrightarrow \ldots
\end{align*} 
where \[H^0(\mathbb{P}^2, \mathcal{O}_{\mathbb{P}^2}(r-s+l))= H^1(\mathbb{P}^2, \mathcal{O}_{\mathbb{P}^2}(r-s+l))=0\]
because $l<s-r$. Therefore, 
\begin{equation}\label{stab}
H^0(\mathbb{P}^2, E(l)) = H^0(\mathbb{P}^2, \mathcal{O}_{\mathbb{P}^2}(r+s+l-t)\otimes I_Z).    \end{equation}
Since $l<s-r,$ and $Z$ is not contained in any curve of degree $2s-1-t$, it follows that $r+s+l-t \leq 2s-1-t$, hence

\[h^0(\mathbb{P}^2, \mathcal{O}_{\mathbb{P}^2}(r+s+l-t)\otimes I_Z) \leq h^0(\mathbb{P}^2, \mathcal{O}_{\mathbb{P}^2}(2s-1-t)\otimes I_Z) \leq 0,\]
we conclude from (\ref{stab}) that $H^0(\mathbb{P}^2, E(l))=0$. This implies that $\mathcal{O}_{\mathbb{P}^2}(-l)$ is not a subbundle of $E$ and $\mathcal{O}_{\mathbb{P}^2}(r-s)$ is maximal, which is the desired conclusion.
\end{proof}

Theorem \ref{invariantedesegre} is given for
stable vector bundles, the following result considers the case for semistable and unstable vector bundles.

\begin{Theorem}\label{invariantedesegrenegativo}  Let $r,c_2, s$
be integer numbers with $s\leq 0$ and $t\in \{0,1\}$. The following statements are equivalent;
\begin{enumerate}
    \item  There exists a rank two vector bundle $E$  with $c_1(E)=2r-t$, $c_2(E)=c_2$
    and $S(E)=2s-t$.
    \item  $c_2\ge r^2-s^2-t(r-s)$.
    \item  For any zero cycle $Z$ that is a locally complete intersection of length $\ell(Z)=c_2+s^2-r^2+t(s-r)$, $Z$ fits into the exact sequence
    \begin{eqnarray*}
     0\rightarrow \mathcal{O}_{\mathbb{P}^2}(r-s)\rightarrow E \rightarrow \mathcal{O}_{\mathbb{P}^2}(r+s-t) \otimes I_Z \rightarrow 0,
    \end{eqnarray*}
where $E$ is a vector bundle.
     \end{enumerate}
\end{Theorem}

\begin{proof} 
$(1)\Rightarrow (2)$ 
Assume that there exists a vector bundle $E$ with
$c_1(E)=2r-t$ $c_2(E)=c_2$ and $S(E)=2s-t$. Therefore, $E$ fits into the following exact sequence 
 \begin{eqnarray*}
    0\rightarrow \mathcal{O}_{\mathbb{P}^2}(r-s)\rightarrow E \rightarrow \mathcal{O}_{\mathbb{P}^2}(r+s-t) \otimes I_Z \rightarrow 0,
    \end{eqnarray*}
where $Z\subset \mathbb{P}^2$ is of codimension $2$. Since $\ell(Z)\ge 0,$ 
it follows that $c_2(E)\ge r^2-s^2-tr+ts.$

$(2)\Rightarrow (3)$
Let $Z\subset \mathbb{P}^2$ be a complete intersection
of codimension two and length $\ell(Z)=c_2+s^2-r^2+rt-st$.
Since $H^0(\mathbb{P}^2,\mathcal{O}_{\mathbb{P}^2}(-t-3))=0$,
it follows that the pair $(\mathcal{O}_{\mathbb{P}^2}(-t-3),Z)$
satisfies the Cayley-Bacharach property. 
Therefore, by Theorem \ref{Cayley-Bacharach} we have the following extension.
\begin{equation}\label{38imaximal}
  0 \longrightarrow \mathcal{O}_{\mathbb{P}^2}(r-s) \longrightarrow E \longrightarrow \mathcal{O}_{\mathbb{P}^2}(r+s-t) \otimes I_Z \longrightarrow 0,
\end{equation}
where $E$ is vector bundle.

$(3)\Rightarrow (1)$. Assume that we have  the exact sequence (\ref{38imaximal}). Therefore $c_1(E)=2r-t$
and $c_2(E) = c_2$. We want to prove that $S(E)=2s-t$. Let $L\subset E$ be a line bundle and consider
$\phi:L\rightarrow \mathcal{O}_{\mathbb{P}^2}(r+s-t)$ as the following composition. 
\begin{eqnarray*}
L\hookrightarrow E \to \mathcal{O}_{\mathbb{P}^2}(r+s-t)\otimes I_Z \hookrightarrow \mathcal{O}_{\mathbb{P}^2}(r+s-t).
\end{eqnarray*}

Now, we consider two cases: 

{\bf Case 1: $\phi=0$.} Then $L\subset \mathcal{O}_{\mathbb{P}^2}(r-s)$
and $c_1(L)\leq r-s.$

\vspace{.1cm}

{\bf Case 2: $\phi\neq 0.$} Then $c_1(L)\leq r+s-t.$ 
Since $t\in \{0,1\}$ and $s\leq 0$, it follows that $c_1(L)\leq r+s-t\leq r+s\leq r-s.$

Hence, $\mathcal{O}_{\mathbb{P}^2}(r-s)$ is maximal and $S(E)=2s-t.$ 

\end{proof}

\section{Coherent Systems.}\label{Sec-CS}

In this section, we recall the main results that we will use on coherent systems.
The results exposed in the first part are general to  coherent systems on any smooth, irreducible surface.
In the second part, we focus to  coherent systems on $\mathbb{P}^2$.

\subsection{Coherent systems on  surfaces} 
Let $X$ be a complex smooth and irreducible projective surface. 
A coherent system on $X$ of type $(n, c_1,c_2, k)$  
is a pair $(\mathcal{E},V)$ where $\mathcal{E}$ is a coherent sheaf
on $X$ of rank $n$, with fixed Chern classes
$c_i \in H^{2i}(X,\mathbb{Z})$, for $i=1,2$ and
$V\subseteq H^{0}(X,\mathcal{E})$ a subspace of dimension $k$. 
A coherent subsystem of $(\mathcal{E},V)$ is a pair $(\mathcal{F},W)$ 
where $0 \neq \mathcal{F}\subset \mathcal{E}$ is a proper subsheaf of 
$\mathcal{E}$ and $W \subseteq V\cap H^{0}(X,\mathcal{F})$. 
When $\mathcal{E}$ is locally free,  we will write $(E,V)$ instead of $(\mathcal{E},V).$

\begin{Definition}\label{DefHPol}
\begin{em}
Let $(\mathcal{E},V)$ be a coherent system on $X$ of type $(n, c_1,c_2,k).$
    Let $\alpha\in \mathbb{Q}[m]$ be a polynomial and let $H$ be an ample line bundle on $X$, 
    the \it{reduced Hilbert polynomial} of $(\mathcal{E},V)$ is defined as:
    \begin{eqnarray*}
p^{\alpha}_{H,(\mathcal{E},V)}(m)=\frac{k}{n}\cdot \alpha +\frac{P_{H,\mathcal{E}}(m)}{n},
\end{eqnarray*}
where $P_{H,\mathcal{E}}(m)$ denotes the Hilbert polynomial of $\mathcal{E}$, i.e., 
\begin{eqnarray*}
\frac{P_{H,\mathcal{E}}(m)}{n} &=&\frac{H^2 m^2}{2} + \left[ \frac{c_1\cdot H}{n}- \frac{K_X\cdot H}{2} \right]m\\
&&+ \frac{1}{n}\left(\frac{c_1^2-(c_1\cdot K_X)}{2}-c_2 \right)+\chi(\mathcal{O}_X).
\end{eqnarray*}
\end{em}
\end{Definition}

Associated to coherent systems and its reduced Hilbert polynomial, a notion of stability is given as follows; For a fixed polynomial $\alpha \in \mathbb{Q}[m]$,
we say that $(\mathcal{E},V)$ is $\alpha$-stable (resp. $\alpha$-semistable)
if for any proper coherent subsystem $(\mathcal{F},W)\subset (\mathcal{E},V)$
we have
\begin{equation}
p^{\alpha}_{H,(\mathcal{F},W)}(m) < p^{\alpha}_{H,(\mathcal{E},V)}(m), \,\,\,\,\,  (resp. \leq).
\end{equation}

The following result shows that it is enough to verify the $\alpha$-stability
of  coherent systems $(E,V)$ 
by examining coherent subsystems $(\mathcal{F},W)$ with $\mathcal{F}$ being a vector bundle.

\begin{Lemma}\label{lemlocfresemi}
Let $H$ be an ample divisor on a smooth, irreducible algebraic surface $X$ and let $\alpha\in \mathbb{Q}[m]$ be a polynomial.
Let $(E,V)$ be a coherent system.
If for every coherent subsystem $(\mathcal{F},W)\subset (E,V)$ with $\mathcal{F}$ a vector bundle,
we have $p_{H,(\mathcal{F},W)}^{\alpha}(m)<p_{H,(E,V)}^{\alpha}(m)$ (resp. $\leq$),
then $(E,V)$ is $\alpha$-stable (resp. $\alpha$-semistable).
\end{Lemma}

\begin{proof}
Let $(\mathcal{F},W) \subset (E,V)$ be a coherent subsystem, 
we will prove that $p_{H,(\mathcal{F},W)}^{\alpha}(m)<p_{H,(E,V)}^{\alpha}(m)$. 
Since $\mathcal{F}$ is torsion free, there exists a canonical embedding
$\mathcal{F}\rightarrow \mathcal{F}^{**}$ which fit in the following diagram 
\begin{equation*}
 \xymatrix{   \mathcal{F}  \ar[r] \ar@{^{(}->}[d]      &  E \ar@{^{(}->}[d] & \\
				 \mathcal{F}^{**} \ar[r]_{}& E^{**}.} \\
\end{equation*}

Therefore, the morphism $\mathcal{F} \rightarrow E$ factors through $\mathcal{F} \hookrightarrow \mathcal{F}^{**}$. The injective morphism $\mathcal{F} \hookrightarrow \mathcal{F}^{**}$ induces an injective morphism $H^0(X,\mathcal{F})\hookrightarrow H^0(X,\mathcal{F}^{**})$.
Therefore, we see that any morphism $(\mathcal{F},W)\rightarrow (E,V)$ factor through $(\mathcal{F}^{**},W)$.  
Therefore, $(\mathcal{F}^{**},W)$ is a coherent subsystem of $(E,V)$.

Since $\mathcal{F}$ is a torsion free sheaf, it follows that $\mathcal{F}$ fits the following exact sequence
\[0 \to \mathcal{F} \to \mathcal{F}^{**} \to \mathcal{F}^{**}/\mathcal{F} \to 0\]
where the support of $\mathcal{F}^{**}/\mathcal{F}$ has codimension two (see \cite[Proposition 1.26]{Brinzanescu}). Therefore, $c_1(\mathcal{F})=c_1(\mathcal{F}^{**})$ and $c_2(\mathcal{F})\geq c_2(\mathcal{F}^{**})$. Since $\text{rk}(\mathcal{F})=\text{rk}(\mathcal{F}^{**})$, it follows that
$p^{\alpha}_{H,(\mathcal{F},W)}(m) \leq p^{\alpha}_{H,(\mathcal{F}^{**},W)}(m)$ for any polynomial $\alpha \in\mathbb{Q}[m]$. The sheaf $\mathcal{F}^{**}$ is reflexive; therefore, it is a vector bundle on $X$. By hypothesis, we have
\begin{eqnarray*}
p^{\alpha}_{H,(\mathcal{F},W)}(m) \leq p^{\alpha}_{H,(\mathcal{F}^{**},W)}(m)<p_{H,(E,V)}^{\alpha}(m),
\end{eqnarray*}
therefore $p_{H,(\mathcal{F},W)}^{\alpha}(m)<p_{H,(E,V)}^{\alpha}(m)$ as we desired.
\end{proof}

\begin{Remark}
\begin{em}
    Notice that a coherent system $(\mathcal{E},V)$ is strictly $\alpha$-semistable 
    if it is $\alpha$-semistable and there exists a coherent subsystem $(\mathcal{F},W)\subset (\mathcal{E},V)$
    such that $p^{\alpha}_{H,(\mathcal{F},W)}(m)= p^{\alpha}_{H,(\mathcal{E},V)}(m)$.
    Moreover, if $(E,V)$ is a coherent system strictly $\alpha$-semistable
    and $(\mathcal{F},W)\subset (E,V)$ a coherent subsystem such that $p^{\alpha}_{H,(\mathcal{F},W)}(m)= p^{\alpha}_{H,(E,V)}(m)$.
    Then, from the proof of the Lemma~\ref{lemlocfresemi} it follows that $\mathcal{F}$ can be considered as locally free.
    \end{em}
\end{Remark}

In \cite{Potier} and \cite{Min} the moduli space
of coherent systems on $X$ is constructed using geometric invariant theory (GIT).
In this paper we will denote by $\mathbf{S}_{X,\alpha}(n,c_1, c_2, k)$ 
the moduli space of $\alpha$-semistable coherent systems of type
$(n,c_1,c_2,k)$ on $X$.

It is known that the property of being locally free is open; here there is an open subscheme 
$\mathbf{M}_{X,\alpha}(n,c_1, c_2, k) \subset \mathbf{S}_{X,\alpha}(n,c_1, c_2, k)$ 
which consists of equivalence classes of $\alpha$-stable coherent systems $(E,V)$
of type $(n, c_1,c_2,k)$. Similarly, we denote by
$\widetilde{\mathbf{M}}_{X,\alpha}(n,c_1, c_2, k) \subset \mathbf{S}_{X,\alpha}(n,c_1, c_2, k)$
the open subscheme of $\alpha$-semistable coherent systems $(E,V)$.

\subsection{Coherent Systems on $\mathbb{P}^2$} 
From now on, we focus our attention on coherent systems $(E,V)$ of type $(2,c_1,c_2,k)$
on $\mathbb{P}^{2}$. Let us denote by
 $(L,W)$  a coherent system on $\mathbb{P}^2$
where $L$ is a line bundle.
Considering the ample line bundle $\mathcal{O}_{\mathbb{P}^2}(1)$ we have that  
the reduced Hilbert polynomial of $(E,V)$ can be expressed as follows:
$$
p^{\alpha}_{(E,V)}(m)=\frac{k}{2}\cdot \alpha +\frac{P_{E}(m)}{2},
$$
where
\begin{equation}
\frac{P_E(m)}{2}=\frac{m^2}{2} + \left[ \frac{c_1+3}{2} \right]m + \frac{1}{2}\left(\frac{c_1^2+3c_1}{2}-c_2 \right)+1.
\end{equation}

If $(L,W)$ is a coherent system of type $(1,c'_1,0,k')$, we have
\begin{eqnarray*}
p^{\alpha}_{(L,W)}(m)=k' \cdot \alpha +P_{L}(m),
\end{eqnarray*}
where
\begin{equation}
P_L(m)=\frac{m^2}{2} + \left[c'_1+\frac{3}{2} \right]m + \frac{(c'_1)^2+3c'_1}{2}+1.
\end{equation}

The $\alpha$-stability of the coherent system depends on $\alpha \in \mathbb{Q}[m]$, 
which can be a constant polynomial, a linear polynomial, 
or a polynomial of degree $d \geq 2$. 
The following studied  conditions of $\alpha$-stability
in each one of these cases.

\begin{Lemma}
Let $\alpha=a\in \mathbb{Q}[m]$ be a constant polynomial,
and let $(E,V)$ be a coherent system of type
$(2,c_1, c_2, k)$ on $\mathbb{P}^2$.
Then $(E,V)$ is $\alpha$-stable (resp. $\alpha$-semistable), if and only if for any coherent subsystem  $(L,W)\subseteq (E,V)$ one of
the following statements hold
\begin{enumerate}
    \item $\mu(E)-\mu(L)>0$.
    \item $\mu(E)-\mu(L)=0$ and \ $4c_2-c_1^2 < 4a (k-2\dim\, W)$ (resp. $\leq$)
\end{enumerate} 
\end{Lemma}

\begin{proof}
Let $\alpha=a\in \mathbb{Q}[m]$ be a constant polynomial. 
Suppose that $(E,V)$ is an $a$-stable coherent system of type
$(2,c_1,c_2,k)$ and let $(L,W)\subseteq (E,V)$ be a coherent subsystem.
Then 
\begin{eqnarray*}
p^a_{(L,W)}(m) = a\cdot \dim W + P_L(m) < p^a_{(E,V)}(m) = \frac{k}{2}a+\frac{P_E(m)}{2}
\end{eqnarray*}


which is equivalent to
\[4mc_1(L)+2[c_1(L)^2+3c_1(L)]+4a \dim W < 2mc_1+ c_1^2+3c_1-2c_2+2a k.\]

Here, we have two cases:

\textbf{Case 1:} The linear coefficient of $P_E(m)/2$ is greater than the linear coefficient of $P_L(m)$, which implies (1).

\textbf{Case 2:} The linear coefficient of $P_E(m)/2$ is equal to the linear coefficient of $P_L(m)$
and the constant part of $p_{(E,V)}^{\alpha}(m)$ is greater than or equal to the constant part of
$p_{(L,W)}^{\alpha}(m)$ which implies $c_1=2c_1(L)$ and
$$
\frac{c_1(L)^{2}+3c_1(L)}{2} + a \dim\, W +1 < \frac{c_1^{2}+3c_1}{4}-\frac{c_2}{2}+a\frac{k}{2}+1,
$$
by the above we have
$$
\frac{c_1^{2}}{8}+\frac{3c_1}{4}+a\dim\, W <\frac{2c_1^{2}+6c_1}{8}-\frac{c_2}{2}+\frac{a k}{2},
$$
equivalently
\begin{eqnarray*}
4c_2-c_1^2< 4a (k-2\dim\, W)
\end{eqnarray*}
 as we desired.
 \end{proof}

\begin{Lemma}\label{lemaconstante}
Let $\alpha\in \mathbb{Q}[m]$ be a polynomial with $\alpha=\sum_{i=0}^{d} a_i m^{i}\in \mathbb{Q}[m],$
$d\geq 2,$ and $a_d\neq 0$. Let $(E,V)$ be a coherent system of type $(2, c_1, c_2, k)$ on $\mathbb{P}^2$.
Then $(E,V)$ is $\alpha$-stable (resp. $\alpha$-semistable) if and only if for any coherent subsystem $(L,W)$ one of the following statements hold
\begin{enumerate}
    \item $\dim\, W<\frac{k}{2}.$
    \item $\dim\, W=\frac{k}{2}$ and $\mu(E)-\mu(L)>0$ (resp. $\geq$).
\end{enumerate} 
\end{Lemma}

\begin{proof}
Notice that the leader coefficient of $p^{\alpha}_{(E,V)}(m)$ is
$k\cdot a_d/2$, the leader coefficient of $p^{\alpha}_{(L,W)}(m)$ is
$(\dim\,W)\cdot a_d$ with $d\geq 2$. Furthermore, 
the linear coefficient of $p^{\alpha}_{(E,V)}(m)$ is $\mu(E)+\frac{a_1\cdot k+3}{2}$
and the linear coefficient of $p^{\alpha}_{(L,W)}(m)$ is
$\mu(L)+\dim\,W\cdot a_1+\frac{3}{2}.$
Therefore, (1) holds if and only if the leader coefficient of 
$p^{\alpha}_{(E,V)}(m)$  is greater than the leader coefficient
of $p^{\alpha}_{(L,W)}(m).$
Now, (2) is equivalent to  that the leader coefficient of
$p^{\alpha}_{(E,V)}(m)$ is equal to the leader coefficient
of $p^{\alpha}_{(L,W)}(m)$ and all other coefficients are equal
(except the constant coefficients may).
\end{proof}

The following result considers the case where $\alpha$ is a linear polynomial.

\begin{Lemma}\label{Lema-Estab-gdo1}
Let $\alpha= am+b\in \mathbb{Q}[m]$ be a linear polynomial with 
$a\neq 0$ and $(E,V)$ a coherent system of type $(2, c_1, c_2, k)$ on $\mathbb{P}^2$.
Then $(E,V)$ is $\alpha$-stable (resp. $\alpha$-semistable) if and only if for any 
coherent subsystem $(L,W)$ one of the following statements hold
\begin{enumerate}
    \item $\mu(E)-\mu(L)>a \left(\dim\, W-\frac{k}{2}\right)$.
    \item $\mu(E)-\mu(L)=a \left(\dim\, W-\frac{k}{2}\right)$
    and 
   
$$
c^{2}_1-2c_2-2c^{2}_1(L)>(4b-6a)\left(\dim\,W- \frac{k}{2}\right) (\mbox{ resp.} \geq).
$$
 
\end{enumerate} 
\end{Lemma}

\begin{proof}
If $(E,V)$ is $\alpha$-stable,
we have two cases.

\textbf{Case 1:} The linear part of $p^{\alpha}_{(E,V)}(m)$ is greater than the linear part
of $p^{\alpha}_{(L,W)}(m)$, which is equivalent to the statement $(1)$.

\textbf{Case 2:} The linear part of $p^{\alpha}_{(E,V)}$ is equal to the
linear part of $p^{\alpha}_{(L,W)}$ and the constant part
of $p^{\alpha}_{(E,V)}(m)$ is greater than the linear part of
$p^{\alpha}_{(L,W)}(m)$. In this case, we have
$\mu(E)-\mu(L)=a \left(\dim\, W-\frac{k}{2}\right)$
and the following inequality.

$$
\frac{c^{2}_{1}+3c_{1}}{4}-\frac{c_{2}}{2}+\frac{b}{2}> \frac{c_1(L)+3c_1(L)}{2}+b (\dim\,W)
$$

equivalently

$$
c_1^{2}+3c_1-2c_2+2b\ > 2c_1(L)^{2}+6c_1(L)+4b\dim\, W
$$

i.e.
$$
c^2_1 + 6\left( \frac{c_1}{2} - c_1(L)\right)-2c_2>2c_1(L)^{2}+4b\left(\dim\, W-\frac{k}{2}\right).
$$

Using $\mu(E)-\mu(L)=a \left(\dim\, W-\frac{k}{2}\right),$ we obtain
$$
c^{2}_1-2c_2-2c^{2}_1(L)>(4b-6a)\left(\dim\,W- \frac{k}{2}\right)
$$
as we desired.
\end{proof}

\section{Segre Invariant and Coherent systems on $\mathbb{P}^{2}$.}\label{sect-SICS}

From now on we consider  coherent systems $(E,V)$ of type $(2;c_1, c_2,k)$ on $\mathbb{P}^2$, in this section we use the Segre invariant  in order to determine properties of semistability. Consider the parameter $t$ with $t=0,1$. From Theorem \ref{invariantedesegre} and Theorem \ref{invariantedesegrenegativo} any vector bundle $E$ of type $(2,2r-t, c_2)$ 
and Segre invariant $S(E)=2s-t$ fits into the following exact sequence
\begin{eqnarray*}
 0\rightarrow \mathcal{O}_{\mathbb{P}^2}(r-s)\rightarrow E \rightarrow \mathcal{O}_{\mathbb{P}^2}(r+s-t) \otimes I_Z \rightarrow 0, 
\end{eqnarray*}
where $Z\subset \mathbb{P}^2$ is of codimension two. Hence from the long exact sequence in cohomology, it follows that  \[H^0(\mathbb{P}^2,E)=H^0(\mathbb{P}^2,\mathcal{O}_{\mathbb{P}^2}(r-s))\oplus H^0(\mathbb{P}^2,\mathcal{O}_{\mathbb{P}^2}(r+s-t)\otimes I_Z).\] 

\begin{Definition}\label{Def-maximalsub}
Let $(E,V)$ be a coherent system. The coherent subsystem $(L_{max},W_{max}) \subset (E,V)$, defined by
$L_{max}=\mathcal{O}_{\mathbb{P}^2}(r-s)$ and $W_{max}=H^0(\mathbb{P}^2,\mathcal{O}_{\mathbb{P}^2}(r-s))\cap V$ is called the maximal subsystem of $(E,V)$.
\end{Definition}

\begin{Definition}
Let $a\in \mathbb{Q}_{+}$ and $(E,V)$ be a coherent system. The slope of $(E,V)$ with respect to $a$, denoted by $\mu_{a}(E,V)$, is defined as                                    \begin{eqnarray*}
\mu_{a}(E,V)=\mu(E)+a\frac{\dim V}{\text{rk}(E)}.
\end{eqnarray*}
where $\mu(E)$ is the usual slope of the vector bundle. 
\end{Definition}

Here are some elementary properties of this concept.

\begin{Lemma}\label{pendientesmaximal}
Let $a\in \mathbb{Q}_{+}$ and let $(E,V)$ be a coherent system. Then, for any subsystem $(L,W)\subset (E,V)$ we have \[\mu_{a}(E,V)- \mu_{a}(L,W)\geq \mu_{a}((E,V))-\mu_{a}(L_{max},W_{max}).\]
\end{Lemma}

\begin{proof}
Assume that the maximal line bundle $L_{max}$ of $E$ is $\mathcal{O}_{\mathbb{P}^2}(d)$, with
$d\in \mathbb{Z}$. Let $(\mathcal{O}_{\mathbb{P}^2}(d_0),W)$ be a coherent subsystem of $(E,V)$, 
since $\mathcal{O}_{\mathbb{P}^2}(d)$ is maximal, it follows that $d_0\leq d$ and there exists an injective morphism
$\mathcal{O}_{\mathbb{P}^2}(d_0)\to \mathcal{O}_{\mathbb{P}^2}(d)$ which defines an injective morphism
$H^0(\mathbb{P}^2,\mathcal{O}_{\mathbb{P}^2}(d_0))\to H^0(\mathbb{P}^2,\mathcal{O}_{\mathbb{P}^2}(d))$ and it concludes the result.
\end{proof}

The interest of the maximal coherent subsystem $(L_{max},W_{max})$ of a coherent system $(E,V)$ is the fact that it determines completely the type of stability of  $(E,V).$

\begin{Corollary}\label{estabilidadconmaximal}
Let $\alpha(m)=am+b$ be a polynomial with $a\in \mathbb{Q}_{+}$ and $(E,V)$ be a coherent system. 
\begin{enumerate}
\item If $\mu_a((E,V))-\mu_{a}((L_{max},W_{max}))<0$, then $(E,V)$ is $\alpha$-unstable.
\item If $\mu_a((E,V))-\mu_{a}((L_{max},W_{max}))>0$, then $(E,V)$ is $\alpha$-stable.
\item If $\mu_a((E,V))-\mu_{a}((L_{max},W_{max}))=0$ and 
\begin{eqnarray*}
c^{2}_1-2c_2-2c^{2}_1(L_{max})> (4b-6a)\left(\dim\,W_{max}- \frac{k}{2}\right)
(resp. \geq),
\end{eqnarray*}
then $(E,V)$ is  $\alpha$-stable (resp. $\alpha$-semistable).
\end{enumerate}
\end{Corollary}

\begin{proof}
It follows from \ref{Lema-Estab-gdo1} $(2)$ and Lemma \ref{pendientesmaximal}.\\
\end{proof}

\begin{Lemma}\label{Lem-c110}
Let $\alpha(m)=am+b$ be a polynomial with $a\in \mathbb{Q}_{+}.$ 
Let $(E,V) $ be a coherent system with $c_1(E)=-t,$ with $ t\in \{0,1\}$ and
$\dim\, V=k\geq 2.$ Then $(E,V)$ is $\alpha$ -semistable if only if $E=\mathcal{O}_{\mathbb{P}^2}\oplus \mathcal{O}_{\mathbb{P}^2}$ and $V=H^0(\mathbb{P}^2,\mathcal{O}_{\mathbb{P}^2}\oplus \mathcal{O}_{\mathbb{P}^2})$. In particular, $ t=0$, $S(E)=0$, and $ c_2(E)=0$.
\end{Lemma}

\begin{proof}
($\Leftarrow$)The coherent system $(\mathcal{O}_{\mathbb{P}^2}\oplus \mathcal{O}_{\mathbb{P}^2},H^0(\mathbb{P}^2,\mathcal{O}_{\mathbb{P}^2}\oplus \mathcal{O}_{\mathbb{P}^2}))$ is $\alpha$-semistable for any polynomial $\alpha(m)=am+b,$ with $a\in \mathbb{Q}_{+}$ by Corollary \ref{estabilidadconmaximal}.

($\Rightarrow$) Assume that $S(E)=2s-t$. Now, if $s\geq 1$ then $E$ is stable. Since $c_1(E)=-t\leq 0,$ it follows that $H^0(\mathbb{P}^2,E)=0,$ which contradicts the hypothesis that $E$ has at least two sections.

Therefore, $s\leq 0,$  and $E$ fits into the following exact sequence
\begin{eqnarray*}
    0\to \mathcal{O}_{\mathbb{P}^2}(-s) \to E\to \mathcal{O}_{\mathbb{P}^2}(s-t)\otimes I_Z\to 0.
\end{eqnarray*}

If $s<0$, since $t\in \{0,1\}$ it follows that $H^0(\mathbb{P}^2,\mathcal{O}_{\mathbb{P}^2}(s-t)\otimes I_Z)=0$ and $(\mathcal{O}_{\mathbb{P}^2}(-s),V)$ is the maximal subsystem of $(E,V)$. Hence, from Corollary \ref{estabilidadconmaximal} (1), we conclude that $(E,V)$
is $\alpha$-unstable which is also a contradiction.  On the other hand, if $s=0$ and $\ell(Z)>0$ or $t=1$, then $h^0(\mathbb{P}^2,E)=1$ which is a contradiction because $\dim\, V\geq 2$.  Therefore, $s=t=\ell(Z)=0$  and since $\text{Ext}^1(\mathcal{O}_{\mathbb{P}^2},\mathcal{O}_{\mathbb{P}^2})=0, $ it follows that $E=\mathcal{O}_{\mathbb{P}^2}\oplus \mathcal{O}_{\mathbb{P}^2}.$
\end{proof}

\begin{Remark}\label{Rem-equality}
\begin{em}
    In particular, if $(E,V)$ is a coherent system with maximal coherent subsystem $(L_{max}. W_{max}).$ Then $p^{\alpha}_{(E,V)}()m= p^{\alpha}_{(L_{max}. W_{max})}(m)$ if and only if the following equality's are satisfied
\begin{eqnarray*}
\mu_a((E,V))-\mu_{a}((L_{max},W_{max}))&=&0\\
c^{2}_1-2c_2-2c^{2}_1(L_{max})- (4b-6a)\left(\dim\,W_{max}- \frac{k}{2}\right)&=&0.
\end{eqnarray*}
\end{em}
\end{Remark}

The following results give $\alpha$ semistability conditions.

\begin{Theorem}\label{chernpositiva}
Let $\alpha(m)=am+b$ be a polynomial with $a\in \mathbb{Q}_{+}.$ 
Let $(E,V) $ be a coherent system of type 
$(2,2r-t,c_2,k)$ with $t\in \{0,1\}$ and
$\dim\, V=k\geq 2.$ If $(E,V)$ is $\alpha$-semistable, then the following conditions are satisfied:
\begin{enumerate}
\item $r\geq 0$.
\item $-r< s\leq r.$ 
\item $c_2\ge 0$.
\end{enumerate}
\end{Theorem}

\begin{proof}
By Corollary \ref{estabilidadconmaximal}, if $(E,V)$ is $\alpha$-semistable, then 
$
\mu_{a}(E,V)-\mu_{a}(L_{max},W_{max})\geq 0.
$
\begin{enumerate}
    \item Assume that $r< 0.$ We have the following cases:

{\bf Case 1: $E$ is  semistable:} 
it follows that $c_1(E)=2r-t<0$, $t \in \{0,1\}$. 
Therefore, $H^0(\mathbb{P}^2,E)=0$, which is a contradiction. 

{\bf Case 2: $E$ unstable:} Then $S(E)=2s-t$ (where $s<0$ if $t=0$ and $s\leq 0$ if $t=1$). From  Theorem \ref{invariantedesegrenegativo}, $E$ fits into the following exact sequence
\begin{eqnarray*}
0 \to \mathcal{O}_{\mathbb{P}^2}(r-s) \to E \to \mathcal{O}_{\mathbb{P}^2}(r+s-t)\otimes I_Z\to 0.
\end{eqnarray*}

 Since $r< 0$, $s<0$ and $t\in \{0,1\}$, it follows that $r+s-t<0$ and $H^0(\mathbb{P}^2,E)=H^0(\mathbb{P}^2,\mathcal{O}_{\mathbb{P}^2}(r-s)).$ Then the maximal subsystem is $(\mathcal{O}_{\mathbb{P}^2}(r-s), V)$ and
 \begin{eqnarray*}
 \mu_a(E,V)-\mu_a(L_{max},W_{max})=-\frac{t}{2}+s-a\frac{k}{2}\leq s-a\frac{k}{2}<0.
 \end{eqnarray*}
 which is a contradiction.
Therefore we conclude that $r \geq 0$.

\item To deal with $-r < s\leq r$  We have the following cases:\\
{\bf Case 1: $s>0$:} We will prove by contradiction that $s\leq r$.\\ 
Assume that $s>r$, From  Theorem \ref{invariantedesegre} $E$ fits into the following exact sequence
\begin{eqnarray*}
0 \to \mathcal{O}_{\mathbb{P}^2}(r-s) \to E \to \mathcal{O}_{\mathbb{P}^2}(r+s-t)\otimes I_Z\to 0, 
\end{eqnarray*}
 where $Z\subset \mathbb{P}^2$ is codimension two and is not contained in a curve of degree $2s-1-t$.
 Since $s>r$ and $t\in \{0,1\}$, it follows that $r+s-t\leq 2s-1-t$,
$H^0(\mathbb{P}^2,\mathcal{O}_{\mathbb{P}^2}(r+s-t)\otimes I_Z)=0$
 and $H^0(\mathbb{P}^2,E)=H^0(\mathbb{P}^2,\mathcal{O}_{\mathbb{P}^2}(r-s))=0$, which is a contradiction.\\

{\bf Case 2: $s\leq 0$:} We will prove by contradiction that $-r \leq s$. Assume that $-r>s$.
From Theorem \ref{invariantedesegrenegativo} $E$ fits into the following exact sequence
\begin{eqnarray*}
0 \to \mathcal{O}_{\mathbb{P}^2}(r-s) \to E \to \mathcal{O}_{\mathbb{P}^2}(r+s-t)\otimes I_Z\to 0, 
\end{eqnarray*}
 where $Z\subset \mathbb{P}^2$ is codimension two. 
 Since $-r>s$ and $t\in \{0,1\}$, it follows that $r+s-t<-t$. 
 Therefore $H^0(\mathbb{P}^2,\mathcal{O}_{\mathbb{P}^2}(r+s-t)\otimes I_Z)=0$ and 
 $H^0(\mathbb{P}^2,E)=H^0(\mathbb{P}^2,\mathcal{O}_{\mathbb{P}^2}(r-s)).$ 
 Then $(\mathcal{O}_{\mathbb{P}^2}(r-s),V)$ is the maximal coherent subsystem of $(E,V)$ and 
 \begin{eqnarray*}
 \mu_a(E,V)-\mu_a(L_{max},W_{max})=-\frac{t}{2}+s-a\frac{k}{2}\leq s-a\frac{k}{2}<0.
 \end{eqnarray*}
Therefore, $(E,V)\notin \widetilde{\mathbf{M}}_{\alpha}(2,c_1,c_2,2).$ 

On the other hand if $s=-r$ then  we have the following exact sequence
\begin{eqnarray*}
0 \to \mathcal{O}_{\mathbb{P}^2}(2r) \to E \to \mathcal{O}_{\mathbb{P}^2}(-t)\otimes I_Z\to 0, 
\end{eqnarray*}
Let $(E,V)$ be a coherent, since $H^0(\mathbb{P}^2,\mathcal{O}_{\mathbb{P}^2}(-t)\otimes I_Z)\leq 1$, it follows that $\dim W_{max}\geq k-1.$ Therefore,
\begin{eqnarray*}
 \mu_a(E,V)-\mu_a(L_{max},W_{max})=-r-\frac{t}{2}+a\left(\frac{k}{2}-(k-1)\right)\leq -r-\frac{t}{2}<0.
 \end{eqnarray*}
 \item Assume that $S(E)=2s-t$. We consider the following cases: $s>0$ and $s\leq 0$.
 
 If $s>0$, from Theorem \ref{invariantedesegre} (2) it follows that $c_2\geq s^2+s+r^2 -t(r+s) >0$.\\
 
 On the other hand, if $s\leq 0$ and $t=0,$ Theorem \ref{invariantedesegrenegativo} (2) implies that
 \begin{eqnarray*}
c_2(E)\geq \text{min}\{ r^2-s^2\mid -r+1\leq s \leq 0\}= 2r-1.
 \end{eqnarray*}
Note that if $r>0$, then $c_2 \geq 0$. 
Now, if $r=0$, then $c_2(E)=0$ by Lemma \ref{Lem-c110}.\\

Now, if $s\leq 0$ and $t=1,$ by Theorem \ref{invariantedesegrenegativo} (2), 
we have $c_2(E)\geq r^2-s^2-r+s$ with $-r+1\leq s\leq 0$ by part $(2)$ of this Theorem.


 Then $s\geq -r+1$ and $r\geq 1$ and we have
 \begin{eqnarray*}
c_2(E)\geq \text{min}\{ r^2-s^2-r+s \mid -r+1\leq  s \leq 0\}=   0.
 \end{eqnarray*}
 then $c_2 \geq 0$. 
\end{enumerate}
\end{proof}


\section{Non emptiness of $\alpha$-semistable coherent system}\label{sect-Nonempt}
 
In this section, we show necessary conditions for the non-emptiness of the moduli spaces $\widetilde{\mathbf{M}_{\alpha}}(2,c_1,c_2, 2)$ and $\mathbf{M}_{\alpha}(2,c_1,c_2, 2).$ The results are divided into two cases: when the first Chern class is even and when the first Chern class is odd.

\begin{Theorem}\label{nonemptinesspar}
 Let $\alpha(m)=am+b\in \mathbb{Q}[m]$ with $a\in \mathbb{Q}_{+}$ and let $c_1=2r-t>0$, $t \in \{0,1\}$. 
Then  $\mathbf{M}_{\alpha}(2,2r-t,c_2,2)\neq \emptyset$ if one of the following conditions is satisfied:
\begin{enumerate}
    \item If $c_2\ge r^2-t+2.$  
    \item If $r^2-rt\leq c_2 <r^2-t+2 ,$ $r\geq 2+t$ and $a>\frac{t}{2}.$
    \item  If $s_0<0$ is an integer such that $t+2-r\leq s_0,$ $\frac{t}{2}-s_0< a$ and $r^2-s_0^2+(s_0-r)t\leq c_2.$ 

\item If $c_2=2r-t, 2r-t-1$ and $a>r-1-\frac{t}{2}$.

\end{enumerate}
\end{Theorem}
\begin{proof}

 Let $\mathbb{L}\subset \mathbb{P}^2$ be a line, and let $Z'\subseteq \mathbb{P}^2$ be a finite set of points.
\begin{enumerate} 
    \item Let  $Z^{'}\subset \mathbb{L}$ with $\ell(Z^{'})=c_2-r^2$  (resp. $\ell(Z^{'})=c_2-r^2+(r-1)$. Now, define $Z_t$ as follows: $Z_0=Z^{'}\cup \{p\},$ with $p\in \mathbb{P}^2\setminus \mathbb{L}$ and $Z_1:=Z^{'}\cup \{p\},$ with  $p\in \mathbb{L}$.  Clearly $\ell(Z_t)=c_2-r^2+t(r-1)+1$ where 
    $Z_0$ is not contained in any curve of degree $1$ and $Z_1$ is contained in $\mathbb{L}.$
    Then, Theorem \ref{invariantedesegre} when $t=0$ (resp. Theorem \ref{invariantedesegrenegativo} when $t=1$), implies that there exists the following exact sequence
    \begin{eqnarray*}
    0\to \mathcal{O}_{\mathbb{P}^2}(r-1)\to E_t \to \mathcal{O}_{\mathbb{P}^2}(r+1-t)\otimes I_{Z_{t}}\to 0,
    \end{eqnarray*}
 where $H^0(\mathbb{P}^2,\mathcal{O}_{\mathbb{P}^2}(r+1-t)\otimes I_{Z_{t}})\geq 1$ and $E_t$ is a vector bundle
    with $c_1(E_t)=2r-t$, $c_2(E)=c_2$ and $S(E_t)=2-t$. \\
    
    Let  $V_t:=W_{t,1}\oplus W_{t,2}$, with 
    $W_{t,1}\subset H^0(\mathbb{P}^2,\mathcal{O}_{\mathbb{P}^2}(r-1))$, 
    $W_{t,2}\subset H^0(\mathbb{P}^2,\mathcal{O}_{\mathbb{P}^2}(r+1-t)\otimes I_{Z_{t}})$ 
    and $ \dim\,W_{t,i}=1.$ 
    Hence, $(\mathcal{O}_{\mathbb{P}^2}(r-1), W_{t,1})$  is the maximal
    coherent subsystem of $(E_t,V_{t})$ and
    
    \begin{eqnarray*}
    \mu_a(E_t,V_{t})-\mu_a(\mathcal{O}_{\mathbb{P}^2}(r-1),W_{t,1})= 1-\frac{t}{2}>0.
    \end{eqnarray*}
    It follows from Corollary \ref{estabilidadconmaximal}, that $(E_t,V_{t})\in \mathbf{M}_{\alpha}(2,2r-t,c_2,2).$\\

\item Let $Z_t\subset \mathbb{L}$ be points with length $\ell(Z_t)=c_2-r^2+tr.$
      By Theorem \ref{invariantedesegrenegativo}(2) with $s=0$, there exists the following exact sequence
    \begin{eqnarray*}
    0\to \mathcal{O}_{\mathbb{P}^2}(r)\to E_{t} \to \mathcal{O}_{\mathbb{P}^2}(r-t)\otimes I_{Z_{t}}\to 0,
    \end{eqnarray*}
    where $E_{t}$ is a vector bundle with $c_1(E_{t})=2r-t$, $c_2(E_{t})=c_2$ and $S(E_{t})=-t.$
Furthermore, since $r\geq 2+t$ and $Z_{t}\subset \mathbb{L}$, we have that 
$h^0(\mathbb{P}^2,\mathcal{O}_{\mathbb{P}^2}(r-t)\otimes I_{Z_{t}})\geq h^0(\mathbb{P}^2,\mathcal{O}_{\mathbb{P}^2}(r-1-t))\geq 3.$ 
 Now, if  $V_{t}\subset H^0(\mathbb{P}^2,\mathcal{O}_{\mathbb{P}^2}(r-t)\otimes I_{Z_{t}})$ with $\dim V_t=2,$ then the maximal subsystem of $(E_t,V_t)$ is $(\mathcal{O}_{\mathbb{P}^2}(r),\mathbf{0})$ and 
    \begin{eqnarray*}
    \mu_a(E_t,V_t)-\mu_a(\mathcal{O}_{\mathbb{P}^2}(r),\mathbf{0})= a-\frac{t}{2}>0,
    \end{eqnarray*}
    if $a>\frac{t}{2}$. By Corollary \ref{estabilidadconmaximal}, $(E,V)\in \mathbf{M}_{\alpha}(2,2r-t,c_2,2).$\\

    \item  Let $Z_t\subset \mathbb{L}$ be points with length $\ell(Z_t)=c_2+t(r-s_0)+s_0^2-r^2$.  Assume that $c_2\geq r^2-s^2_0+t(r-s_0)$ and $s_0<0$. Now, Theorem \ref{invariantedesegrenegativo}(2) implies that there exists an exact sequence
    
    $$
    0\rightarrow \mathcal{O}_{\mathbb{P}^2}(r-s_0)\rightarrow E_t\rightarrow \mathcal{O}_{\mathbb{P}^2}(r+s_0-t)\otimes I_{Z_t}\rightarrow 0.
    $$
    
    Using $r+s_0-t\geq 2$ and since that  $Z_t \subset \mathbb{L}$ , we have $h^{0}(\mathcal{O}_{\mathbb{P}^2}(r+s_0-t)\otimes I_{Z_t})\geq 2.$ Therefore, we take a subspace $V_t\subseteq H^0(\mathbb{P}^2,\mathcal{O}_{\mathbb{P}^2}(r+s_0-t)\otimes I_{Z_{t}}$ with $\dim V_t=2.$

    Thus, $(E_t,V_t)$ is of type $(2,2r-t,c_2,2)$ with maximal coherent subsystem  $(\mathcal{O}_{\mathbb{P}^2}(r-s_0),\mathbf{0}).$ Moreover,

    $$
    \mu_a(E_t,V_t)-\mu_a(\mathcal{O}_{\mathbb{P}^2}(r-s_0),\mathbf{0})=s_0-\frac{t}{2}+a >0
    $$
    
   if $a>\frac{t}{2}-s_0$ and by Corollary \ref{estabilidadconmaximal}, $(E_t,V_t)\in \mathbf{M}_{\alpha}(2,2r-t,c_2,2).$ \\

\item Now, we want to study the case

\begin{eqnarray}\label{ecuacionrest1}
r^2-s_0^2+(s_0-r)t\leq c_2<r^2-(s_0+1)^2 +(s_0+1-r)t
\end{eqnarray}
with $s_0=t+1-r.$ This is equivalent to
\begin{eqnarray}\label{ecuacionrest2}
-t-1+2r\leq c_2< 4r-4-2t.
\end{eqnarray}

\noindent Hence considering $j\in \mathbb{Z},$ we have the following statements:
\begin{enumerate}
\item If $j\geq 1$, then $r^2-(s_0+j)^2+(s_0+j-r)t\geq r^2-(s_0+1)^2+(s_0+1-r)t>c_2.$
\item If $j\leq 0$, then $c_2\geq r^2-s_0^2+ (s_0-r)t\geq r^2-(s_0+j)^2+(s_0+j-r)t.\\$
\end{enumerate}

Now by (\ref{ecuacionrest1}) and Theorem \ref{invariantedesegrenegativo}, there exists an extension
  \begin{eqnarray*}
    0\to \mathcal{O}_{\mathbb{P}^2}(r-s_0-j)\to E_t \to \mathcal{O}_{\mathbb{P}^2}(r+s_0+j-t)\otimes I_{Z_{t}}\to 0
    \end{eqnarray*}

   with $\ell(Z_{t})=c_2-(2r-1-t-j)(1+j)$, if and only if, $j\leq 0$. Given that $s_{0}=t+1-r$, the extension can be described as follows
  \begin{eqnarray*}
    0\to \mathcal{O}_{\mathbb{P}^2}(2r-1-t-j))\to E_{t} \to \mathcal{O}_{\mathbb{P}^2}(1+j)\otimes I_{Z_t}\to 0.
    \end{eqnarray*}
   
    Moreover, $c_2(E_t)=c_2,$  and $\mathcal{O}_{\mathbb{P}^2}(2r-1-j)\subset E$ maximal subbundle. Notice that $h^0(\mathbb{P}^2,\mathcal{O}_{\mathbb{P}^2}(1+j)\otimes I_{Z_t}\geq 2$  if and only if $j=0$ and $\ell(Z_t)\leq 1.$
    The condition $\ell(Z{_{t}})=1$ is equivalent to $c_2= 2r-t$ and $\ell(Z{_{t}})=1$ is equivalent to  $c_2=2r-1-t.$\\

In this case, we take  a coherent system $(E_t,V_t)$ with $V_t\subseteq H^0(\mathbb{P}^2,\mathcal{O}_{\mathbb{P}^2}(1)\otimes I_{Z_t})$ of dimension two. Hence,  $W_{max}=\mathbf{0}$ and 
\begin{eqnarray*}
\mu_a(E_t,V_t)-\mu_a(\mathcal{O}_{\mathbb{P}^2}(2r-1-t),\mathbf{0})=1-r+a +\frac{t}{2}> 0
\end{eqnarray*}
where the last inequality holds if and only if  $a> r-1-\frac{t}{2}.$ 

\end{enumerate}

\end{proof}

\begin{Remark}\label{Rem-unst}
\begin{em}
     Notice that if $(E_{t},V_t)\in \mathbf{M}_{\alpha}(2,2r-t,c_2, 2)$
      with $S(E_t)=2s-t<0$ and its maximal coherent subsystem is $(L_{t, max}, W_{t, max})$, then $dim\,W_{t,max}=0$.
      This is because if $\dim W_{t,max}\ge 1,$ then 
    \begin{eqnarray*}
    \mu_a(E_{t},V_{t})-\mu_a(L_{t,max},W_{t,max})=s-\frac{t}{2}+a(1-\dim W_{t, max})\leq s-\frac{t}{2}<0,
    \end{eqnarray*}
    i.e.,  $(E_t,V_t)\notin  \mathbf{M}_{\alpha}(2,2r-t,c_2,2)$ for any $\alpha$. 
    
    Hence, to prove $\mathbf{M}_{\alpha}(2,2r-t,c_2, 2)\neq \emptyset$, it is necessary to establish the existence of a vector bundle $E_t$ that fits into the exact sequence.
    
    \[ 0 \rightarrow L_{t,max} \rightarrow E_t \rightarrow Q \otimes I_{Z_t} \rightarrow 0 \] 
    
    where $Q \otimes I_{Z_t}$ should have at least two sections. Note that if $(E_t,V_t)$ is a coherent system such that $V_t\subseteq H^{0}(\mathbb{P}^2,Q \otimes I_{Z_t})$ and $(L_{t,max},W_{t,max})=(L_{t,max},\mathbf{0})$ is its maximal coherent subsystem then
 $$
\mu_a(E_t,V_t)-\mu_a(L_{t,max},\mathbf{0})=s-\frac{t}{2}+a.  $$

Therefore, if $a>-s+\frac{t}{2}$ then $(E_t,V_t)\in \mathbf{M}_{\alpha}(2,2r-t,c_2,2)$
and thus providing conditions for nonemptiness of $\mathbf{M}_{\alpha}(2,2r-t,c_2,2)$.
    \end{em}
\end{Remark}

\begin{Proposition} Let $\alpha(m)=am+b$ be a linear polynomial with $a\in \mathbb{Q}_{+}.$ Let $r,c_2$ be a natural number with $t\in \{0,1\}$  and 
$$r^2-s_0^2+(s_0-r)t\leq c_2<r^2-(s_0+1)^2 +(s_0-r+1)t,$$
  where $s_0$ is an integer such that $t+2-r\leq s_0<0.$ 
   Then
    $\mathbf{M}_{\alpha}(2, 2r-t, c_2,2)\neq \emptyset$ (resp. $\widetilde{\mathbf{M}}_{\alpha}(2, 2r-t, c_2,2)\neq \emptyset$) if and only if satisfies one the following statements
    \begin{enumerate}
\item[(1)] $a>-s_0+\frac{t}{2}.$ (resp. $\geq$) 
\item[(2)] $a=-s_0+\frac{t}{2}$
 and 
$$ b> \frac{c_2}{2}-\frac{r^2}{2}+rt-rs_0 +\frac{s_0^2}{2}-\frac{t^2}{4}-\frac{3}{2}s_0+\frac{3}{4}t \hspace{.3cm} (resp. \geq ).$$

 \end{enumerate}
    \end{Proposition}
    
\begin{proof}
Note that if
\begin{equation*}
    r^2-s_0^2+(s_0-r)t\leq c_2<r^2-(s_0+1)^2 +(s_0-r+1)t
\end{equation*}
then $r^2-(s_0+j)^2+(s_0+j-r)t\leq r^2-s_0^2+(s_0-r)t\leq c_2$ for $j\leq 0$ and 
$c_2<r^2-(s_0+1)^2 +(s_0-r+1)t< r^2-(s_0+j +1)^2 +(s_0+j -r+1)t$ if $j>0$.
Now, by Theorem \ref{invariantedesegrenegativo} there exists an extension
  \begin{eqnarray*}
    0\to \mathcal{O}_{\mathbb{P}^2}(r-s_0-j)\to E_{j,t} \to \mathcal{O}_{\mathbb{P}^2}(r+s_0+j-t)\otimes I_{Z_{t}}\to 0
    \end{eqnarray*}   
    with $\ell(Z_{t})=c_2-(r-s_0-j)(r+s_0+j-t) $ and $\mathcal{O}_{\mathbb{P}^2}(r-s_0-j)\subset E_{j,t}$ maximal subbundle, if and only if, $j\leq 0$. Moreover, $S(E_{j,t})=2s_0+2j-t<0,$ i.e., $E_{j,t}$ is unstable as vector bundle. \\

    Now, $h^{0}(\mathbb{P}^2,\mathcal{O}_{\mathbb{P}^2}(r+s_0+j-t))\geq 2$ because $t+2-r\leq s_0<0.$
    Then we take $V_{j,t}\subseteq H^{0}(\mathbb{P}^2,\mathcal{O}_{\mathbb{P}^2}(r+s_0+j-t))$ such that $\dim V_{j,t}=2.$ Hence, we have a coherent system $(E_{j,t},V_{j,t})$ where  
 $(\mathcal{O}_{\mathbb{P}^2}(r-s_0-j),\mathbf{0})$
     is its  maximal coherent subsystem and

     \begin{eqnarray}\label{diferenciapendientes}
    \mu_a(E_{j,t},V_{j,t})-\mu_a(\mathcal{O}_{\mathbb{P}^2}(r-s_0-j),\mathbf{0})=a+s_0+j-\frac{t}{2}.
\end{eqnarray}
We will use these coherent systems $(E_{j,t}, V_{j,t})$ to prove our result.

    $(\Rightarrow)$ 
 We prove that the non-emptiness of $\mathbf{M}_{\alpha}(2,2r-t,c_2,2)$ implies either $(1)$ or $(2)$.
 
 \begin{enumerate}
     \item Assume that $a<-s_0+\frac{t}{2}$. Since $j\leq 0$, we have that 
        \begin{eqnarray*}
    \mu_a(E_{j,t},V_{j,t})-\mu_a(\mathcal{O}_{\mathbb{P}^2}(r-s_0-j),\mathbf{0})&=&a+s_0+j-\frac{t}{2}<0
\end{eqnarray*}
for any $j$. Therefore $(E_{j,t},V_{j,t})\notin \mathbf{M}_{\alpha}(2,2r-t,c_2,2).$\\

\item Assume that $a=-s_0+\frac{t}{2}$ and $b\leq \frac{c_2}{2}-\frac{r^2}{2}+rt-rs_0 +\frac{s_0^2}{2}-\frac{t^2}{4}-\frac{3}{2}s_0+\frac{3}{4}t .$  
Hence
\begin{eqnarray*}
    \mu_a(E_{j,t},V_{j,t})-\mu_a(\mathcal{O}_{\mathbb{P}^2}(r-s_0-j),\mathbf{0})&=&a+s_0+j-\frac{t}{2}=j\leq 0.
\end{eqnarray*}
Therefore, if $j<0$, then $(E_{j,t},V_{j,t})\notin \mathbf{M}_{\alpha}(2,2r-t,c_2,2).$
On the other hand if $j=0$,
and since 
$$ b\leq  \frac{c_2}{2}-\frac{r^2}{2}+rt-rs_0 +\frac{s_0^2}{2}-\frac{t^2}{4}-\frac{3}{2}s_0+\frac{3t}{4}$$
is equivalent to
$$
(2r-t)^2-2c_2-2(r-s_0)^2\leq 6a-4b.
$$
Then,  Corollary \ref{estabilidadconmaximal} implies $(E_{0,t},V_{0,t})\notin \mathbf{M}_{\alpha}(2,2r-t,c_2,2)$.\\

 \end{enumerate}
 
Consequently, we conclude that $\mathbf{M}_{\alpha}(2,2r-t,c_2,2) =\emptyset.$ Similarly to prove the $\alpha-$semistable case.
\vspace{.2cm}

 $(\Leftarrow)$ Suppose $(1)$ or $(2)$ and we will prove $\mathbf{M}_{\alpha}(2,2r-t,2)
\neq \emptyset$.
 As observed at the begin of the proof, there is a coherent system $(E_{0,t},V_{0,t})$ whose maximal coherent subsystem is $(\mathcal{O}(r-s_0),\mathbf{0}).$ Now if
$a> -s_0+\frac{t}{2}$ then
\begin{equation*}
\mu_a(E_{0,t},V_{0,t})-\mu_a(\mathcal{O}_{\mathbb{P}^2}(r-s_0),\mathbf{0})=a+s_0-\frac{t}{2}>0
\end{equation*}
leading to the conclusion that $(E_{0,t},V_{0,t})\in \mathbf{M}_{\alpha}(2,2r-t,2).$
On the other hand, if  $a=-s_0+\frac{t}{2}$, the condition on $b$ and Corollary \ref{estabilidadconmaximal} imply that 
$(E_{0,t},V_{0,t})\in \mathbf{M}_{\alpha}(2,2r-t,2).$ Similarly to prove the $\alpha-$semistable case.\\

\end{proof}

\section{critical values}\label{sect- CV} 
 The moduli space $\mathbf{M}_{\alpha}(n,c_1,c_2;k)$ of $\alpha$-stable coherent systems of type $(n;c_1,c_2;k)$ depends on a parameter $\alpha \in \mathbb{Q}[m]$ and does not change between critical values.  The aim of this section is to compute some critical values in order to describe the changes between consecutive critical values. Note that we consider the usual lexicographic order on the set of polynomials $\alpha \in \mathbb{Q}[m]$.

\begin{Definition}(regular and critical value) 
Let $\gamma >0$ be a polynomial.  We say that $\gamma$ is a regular value if there exist polynomials
$0 < \alpha_1,\alpha_2$ with $\alpha_1<\gamma<\alpha_2$ such that for any $\alpha\in (\alpha_1,\alpha_2)$, 
the moduli spaces $\mathbf{M}_{\alpha}(2,c_1,c_2,2)$ are isomorphic. Otherwise, we say that $\gamma$ is a critical value.
\end{Definition}

The following result gives a charaterization of critical values  using the Hilbert polynomial, compare this with Section 4.2 of \cite{Min}.

\begin{Lemma}\label{nocritical}
Let $0<\alpha(m)=am+b\in \mathbb{Q}[m]$. If $p^{\alpha}_{(L_{max},W_{max})}(m)<p^{\alpha}_{(E,V)}(m)$ for any $(E,V)\in \widetilde{\mathbf{M}}_{\alpha}(2,c_1,c_2,2)$ then $\alpha$  is a regular value.
 \end{Lemma}

\begin{proof}
Let $0 < \alpha(m)=am+b$. We will show that there exist $\alpha_1, \alpha_2\in \mathbb{Q}[m]$ such that the moduli spaces $\widetilde{\mathbf{M}}_{\alpha}(2,c_1,c_2,2)$ are isomorphic for any $\alpha \in (\alpha_1,\alpha_2)$. Since $p^{\alpha}_{(L_{max},W_{max})}(m)< p^{\alpha}_{(E,V)}(m)$ for any $(E,V)\in \widetilde{\mathbf{M}}_{\alpha}(2,c_1,c_2,2)$, 
Corollary \ref{estabilidadconmaximal} shows that  one of the following inequalities is satisfied;
\begin{enumerate}
\item $\mu_a(E,V)-\mu_a(L_{max},W_{max})>0$,
\item  $\mu_a(E,V)-\mu_a(L_{max},W_{max})=0 $ and
$c^{2}_1-2c_2-2c^{2}_1(L_{max})> (4b-6a)\left(\dim\,W_{max}- \frac{k}{2}\right).$ 
\end{enumerate}

If $0<\mu_a(E,V)-\mu_a(L_{max},W_{max})$, then $(E,V)\in \widetilde{\mathbf{{M}}}_{\beta}(2,c_1,c_2,k)$ 
for any $\beta(m)=am+b_1 \in  \mathbb{Q}[m]$.  On the other hand, if $\mu_a(E,V)-\mu_a(L_{max},W_{max})=0 $ and $c^{2}_1-2c_2-2c^{2}_1(L_{max})> (4b-6a)\left(\dim\,W_{max}- \frac{k}{2}\right),$ then we have the following  cases depending on $\dim\, W_{max}$; 

\begin{itemize}
\item[(a)] If $\dim W_{max}=\frac{k}{2},$ then $c^{2}_1-2c_2-2c^{2}_1(L_{max})> 0$.
Therefore, $(E,V)\in \widetilde{\mathbf{M}}_{\beta}(2,c_1,c_2,k)$ for any $\beta(m)=am+b_1 \in \mathbb{Q}[m]$.

\item[(b)] If   $\dim\, W_{max}>\frac{k}{2},$ then
\begin{eqnarray*}
\frac{(c^{2}_1-2c_2-2c^{2}_1(L_{max}))+6a(\dim W_{max}-\frac{k}{2})}{4(\dim W_{max}-\frac{k}{2})} > b.
\end{eqnarray*}

From Theorem \ref{chernpositiva} we have $-r < s \leq r$ where $S(E)=2s-t$, $t=0,1$. Hence  $-r< c_1-2c_1(Lmax)<2r-t$ and therefore $c_1(L_{max})$ and $\dim \, W_{max}$ are bounded. It follows that
\[\frac{(c^{2}_1-2c_2-2c^{2}_1(L_{max}))+6a(\dim W_{max}-\frac{k}{2})}{4(\dim W_{max}-\frac{k}{2})}\] 
takes a finite number of values.

Define $b_{\min}\in \mathbb{Q}$ as
$$
b_{\min}:=\min \left\{ \frac{(c^{2}_1-2c_2-2c^{2}_1(L_{max}))+6a(\dim W_{max}-\frac{k}{2})}{4(\dim W_{max}-\frac{k}{2})}\right\}.
$$
If $b_1\in \mathbb{Q}$ is such that $b_1<b_{\min}$, then $(E,V)\in \widetilde{\mathbf{M}}_{\beta}(2,c_1,c_2,2)$ with $\beta(m)=am+b_1$.\\

\item[(c)] If  $\dim W_{max}<\frac{k}{2},$ we have
\begin{eqnarray*}
\frac{(c^{2}_1-2c_2-2c^{2}_1(L_{max}))+6a(\dim W_{max}-\frac{k}{2})}{4(\dim W_{max}-\frac{k}{2})
} < b.
\end{eqnarray*}

According to the previous item it is not hard to check that
$$
\frac{(c^{2}_1-2c_2-2c^{2}_1(L_{max}))+6a(\dim W_{max}-\frac{k}{2})}{4(\dim W_{max}-\frac{k}{2})}
$$
takes a finite number of values.
Define $b_{\max}\in \mathbb{Q}$ as

$$
b_{\max}:=\max  \left\{ \frac{(c^{2}_1-2c_2-2c^{2}_1(L_{max}))+6a(\dim W_{max}-\frac{k}{2})}{4(\dim W_{max}-\frac{k}{2})}\right\}.
$$

If $b_2\in \mathbb{Q}$ is such that  $b_2> b_{\max},$ then $(E,V)\in \widetilde{\mathbf{M}}_{\beta}(2,c_1,c_2,2)$ with $\beta(m)=am+b_2$.\\

\end{itemize}
Considering $\alpha_1:= am+b_{max}$ and $\alpha_2:= am+b_{\min}$ we have $\alpha\in (\alpha_1 , \alpha_2).$ Moreover, if $(E,V)\in \widetilde{\mathbf{M}}_{\alpha}(2,c_1,c_2,2)$
then
$(E,V)\in \widetilde{\mathbf{M}}_{\beta}(2,c_1,c_2,2)$
for any $\beta\in (\alpha_1 ,\alpha_2)$. This implies that $\alpha$ is a regular value.
\end{proof}

Min He describes the critical values of coherent systems on surfaces and
these critical values can be computed in terms of $\alpha$-semistability (see \cite[Section 4.2]{Min}).  We describe this relation explicitly in the following results.\\

\begin{Lemma}\label{critical}
Let $0< \alpha(m)=am+b\in \mathbb{Q}[m]$.
Assume that there exists a coherent system $(E,V)\in \widetilde{\mathbf{M}}_{\alpha}(2,c_1,c_2,k)$ with  maximal coherent subsystem $(L_{max}, W_{max})$  such that

$
p^{\alpha}_{(L_{max},W_{max})}(m)= p^{\alpha}_{(E,V)}(m),  
$
and
$dim\, W_{max}\neq\frac{k}{2}$. Then $\alpha(m)$ is a critical value.
 \end{Lemma}

\begin{proof}
Let $0< \alpha(m)=am+b\in \mathbb{Q}[m]$. and let $(E,V)\in \widetilde{\mathbf{M}}_{\alpha}(2,c_1,c_2,k)$ with  maximal coherent subsystem $(L_{max}, W_{max})$ such that 
$
p^{\alpha}_{(L_{max},W_{max})}= p^{\alpha}_{(E,V)}.
$
Equivalently, $(E,V)$ satisfies the following equations:
\begin{enumerate}
    \item $\mu_a(E,V)-\mu_a(L_{max},W_{max})=0$;
    \item $c^{2}_1-2c_2-2c^{2}_1(L_{max})= (4b-6a)\left(\dim\,W_{max}- \frac{k}{2}\right)$.
\end{enumerate}

We will prove that $\alpha $ is a critical value.  Let $b_1,b_2\in \mathbb{Q}$ be such that $b_1<b<b_2$ and let $\alpha_1(m)=am+b_1,$ $\alpha_2(m)=am+b_2 \in \mathbb{Q}[m]$, hence $\alpha\in (\alpha_1,\alpha_2)$. We have the following cases depending on $\dim\, W_{max}$:

\begin{enumerate}
    \item Suppose that $\dim\,W_{max}- \frac{k}{2}>0$. It follows that 
\begin{eqnarray*}
c^{2}_1-2c_2-2c^{2}_1(L_{max})&>& (4b_1-6a)\left(\dim\,W_{max}- \frac{k}{2}\right).    
\end{eqnarray*}
and
\begin{eqnarray*}
c^{2}_1-2c_2-2c^{2}_1(L_{max})&<& (4b_2-6a)\left(\dim\,W_{max}- \frac{k}{2}\right).    
\end{eqnarray*}

From Corollary \ref{estabilidadconmaximal},  we have $(E,V)\in  \widetilde{\mathbf{M}}_{\alpha_1}(2,c_1,c_2,k)$ and $(E,V)\notin \widetilde{\mathbf{M}}_{\alpha_2}(2,c_1,c_2,k).$

\item If  $\dim\,W_{max}- \frac{k}{2}<0$ we have

\begin{eqnarray*}
c^{2}_1-2c_2-2c^{2}_1(L_{max})&<& (4b_1-6a)\left(\dim\,W_{max}- \frac{k}{2}\right).    
\end{eqnarray*}
and
\begin{eqnarray*}
c^{2}_1-2c_2-2c^{2}_1(L_{max})&>& (4b_2-6a)\left(\dim\,W_{max}- \frac{k}{2}\right).    
\end{eqnarray*}
By Corollary \ref{estabilidadconmaximal}, we have $(E,V)\notin \widetilde{\mathbf{M}}_{\alpha_1}(2,c_1,c_2,k)$ and $(E,V)\in \widetilde{\mathbf{M}}_{\alpha_2}(2,c_1,c_2,k).$
\end{enumerate}
Therefore, $\alpha $ is a critical value which is the desired conclusion.
\end{proof}

\begin{Remark}\label{Rem-zero}
    \begin{em}
    Assume $0<\alpha(m)=am+b.$ Notice that if $(E,V)\in \widetilde{\mathbf{M}}_{\alpha}(2,c_1,c_2,k)$ with maximal coherent subsystem $(L_{max},W_{max})$ such that $\dim W_{max} = \frac{k}{2}$ and $p^{\alpha}_{(L_{max},W_{max})}= p^{\alpha}_{(E,V)},$ then $S(E)=0,$ $c_1=2c_1(L_{max})$ and $c_2=c_1^2(L_{max}).$ Therefore, $(E,V)\in \widetilde{\mathbf{M}}_{\beta}(2,2r,r^2,k)$
for any polynomial $\beta(m)=\alpha(m)+b',$ with $b'\in \mathbb{Q}$.\\
\end{em}
\end{Remark}

\begin{Corollary}
\label{equivcritical}
Let $0 <\alpha(m)=am+b\in \mathbb{Q}[m]$. Then, $\alpha$ is a critical value if and only if there exists a coherent system
$(E,V)\in \widetilde{\mathbf{M}}_{\alpha}(2,c_1,c_2,k)$  with maximal coherent subsystem $(L_{max},W_{max})$ such that
$p^{\alpha}_{(L_{max},W_{max})}(m)= p^{\alpha}_{(E,V)}(m)$ and 
$dim\,W_{max}\neq \frac{k}{2}$.
\end{Corollary}
\begin{proof}
It follows from Lemma \ref{nocritical} and Lemma \ref{critical}.
\end{proof}

We now determine the critical values for $k=2$. Here, the critical values depend on $c_1$ and are strongly dependent on $c_2$. Consequently, we first present three theorems that consider $c_1$ being even and vary according to $c_2$.
Subsequently, we give the corresponding results for $c_1$ odd.

\begin{Remark}\label{Remnotcritreg}
\begin{em}
Let $0<\alpha(m)=am+b,$ notice that if $(E,V)\in \widetilde{\mathbf{M}}_{\alpha}(2,2r,c_2,2)$ with $S(E)=0$ and its maximal coherent subsystem $(L_{max},W_{max})$
is such that $dim\,W_{max}\neq 1$ then $S(E)=2s$ with $s\in \mathbb{Z}$ and
$(E,V)$ does not satisfy Corollary \ref{equivcritical} because
\begin{enumerate}
    \item If $\dim\, W_{max}=2$, then
$ \mu_a(E,V)-\mu_a(L_{max},W_{max})= s+a(1-\dim W_{max})<0.$

\item If $\dim\,W_{max}=0,$ then 
$
\mu_a(E,V)-\mu_a(L_{max},W_{max})= s+a(1-\dim W_{max})>0.
$
\end{enumerate}
Therefore such $(E,V)$ does not determine if $\alpha$ is critical or regular value.\\
\end{em}
\end{Remark}


\begin{Theorem}\label{criticalThm1}
Let $\widetilde{\mathbf{M}}_{\alpha}(2,c_1,c_2,2)$  be the moduli space of
$\alpha$-semistable coherent systems, $0<\alpha(m)=am+b\in \mathbb{Q}[m]$ and $c_1=2r$. If $c_2\ge r^2+2,$ then $\alpha$ is a critical value if and only if $a$ is a natural number that satisfies one of the following conditions:
\begin{enumerate}
\item
$
a\in \{s\in\mathbb{N}\ \vert \ s<r  \text{ and } c_2\geq r^2+s^2+s \}
$
and
\begin{eqnarray*}
b=\frac{1}{2}r^2-\frac{1}{2}c_2+ar-\frac{1}{2}a^2+\frac{3}{2}a.
\end{eqnarray*}
\item $ a=-s \mbox{ with } s\geq -r+2$
and 
\begin{eqnarray*}
b=\frac{c_2}{2}-\frac{r^2}{2}+ra+\frac{a^2}{2}+\frac{3}{2}a.\\
\end{eqnarray*}

\end{enumerate}

\end{Theorem}

\begin{proof}

According to Remark \ref{Rem-zero} and Remark \ref{Remnotcritreg}, the proof of the theorem is restricted to coherent systems $(E,V) \in \widetilde{\mathbf{M}}_{\alpha}(2,2r,c_2,2)$ where $S(E)=2s \neq 0$
and is presented in two part, when $s>0$ and $s<0$.

   \begin{enumerate}
               \item  $s>0$. 
                \begin{enumerate}
                      \item If $\text{dim}W_{max}\leq 1$, we get 
                            \[
                             \mu_a(E,V)-\mu_a(L_{max},W_{max})= s+a(1-\dim W_{max})\geq s>0
                            \]
                            and $(E,V)$ does not satisfy the hypothesis of Corollary \ref{equivcritical}. 

                      \item If  $\dim W_{max}=2$,  we have 
                            \[
                            \mu_a(E,V)-\mu_a(L_{max},W_{max})= s-a.
                            \]
Hence, in order to satisfy the Corollary \ref{equivcritical} we need $a=s$ and $L_{max}=\mathcal{O}_{\mathbb{P}^2}(r-s).$ Moreover, by Theorem \ref{invariantedesegre}(1), a stable vector bundle $E$ with
$S(E)=2s>0$ exists if and only if $c_2\geq r^2 +s^2+s,$  and by  Theorem \ref{chernpositiva}(2) we have $s\leq r.$ However, if $s=r$, then $L_{max} = \mathcal{O}$ and $\dim W_{max}\neq 2.$
Thus, such coherent system can exist if $a$ is a positive integer that fulfills the following condition

$$
a\in \{s\in\mathbb{N}\ \vert \ s<r  \text{ and } c_2\geq r^2+s^2+s \}
$$

Considering the above conditions for $a$, we now obtain the expression for $b$. We know that the equality $ p^{\alpha}(E,V)=p^{\alpha}(L_{max},W_{max})$ implies that
\begin{eqnarray}\label{bforpositives}
c_1^2-2c_2-2(r-s)^2=(4b-6a)(\dim W_ {max}-1).
\end{eqnarray}
Given that $a=s,$ $\dim W_{max}=2$ and $c_1=2r$ then
$$
b=\frac{r^2-a^2+2ra +3a-c_2}{2}.\\
$$

           \end{enumerate}

\item  $s<0$. 
         \begin{enumerate}
               \item If $\dim W_{max}\geq 1,$ we get
                   \begin{eqnarray*}
                    \mu_a(E,V)-\mu_a(L_{max},W_{max})= s +a(1-\dim W_{max}) \leq  s <0,
                    \end{eqnarray*}
                     hence $(E,V)$ does not meet with hypothesis of Corollary \ref{equivcritical}.

              \item If $\dim W_{max}=0$, then
                     \begin{eqnarray*}
                      \mu_a(E,V)-\mu_a(L_{max},W_{max})&=& s +a (1-\dim W_{max}),\\
                      &=& s+a,
               \end{eqnarray*}
 and we get that $a=-s$. Now, from Theorem ~\ref{chernpositiva}(2) we have $s\geq -r+1, $
 therefore $a\leq r-1.$  Assume that $s=-r+1,$ therefore $E$ fits into the following exact sequence 
 \begin{eqnarray*}
0\to \mathcal{O}_{\mathbb{P}^2}(2r-1)\to E\to \mathcal{O}_{\mathbb{P}^2}(1)\otimes I_Z\to 0
 \end{eqnarray*}
     consequently $\dim W_{max}=0$ if and only if $\ell(Z)=0,1$ or equivalent $c_2(E)=2r,2r-1.$ 
     However, by hypothesis, $c_2\geq r^2+2,$ which implies $c_2(E)>2r$ and then we have no critical points in this case.
     Therefore $a=-s$ with $s\geq -r+2.$ Now, we see the condition for $b$: since $p^{\alpha}(E,V)=p^{\alpha}(L_{max},W_{max})$,
     it follows that
\begin{eqnarray*}
c^{2}_1-2c_2-2c^{2}_1(L_{max})=(4b-6a)\left(\dim\,W_{max}- 1\right)=6a-4b,
\end{eqnarray*}
that is
\begin{eqnarray*}
b=\frac{c_2}{2}-\frac{r^2}{2}+ra+\frac{a^2}{2}+\frac{3}{2}a.\\
\end{eqnarray*}

\end{enumerate}
\end{enumerate}

\end{proof}

\begin{Theorem}\label{criticalThm2}
Let $\widetilde{\mathbf{M}}_{\alpha}(2,c_1,c_2,2)$  be the moduli space of
$\alpha$-semistable coherent systems, $0<\alpha(m)=am+b\in \mathbb{Q}[m]$ and $c_1=2r$.  If $r^2\le c_2<r^2+2$, then $\alpha$ is a critical value if and only if 
\begin{eqnarray*}
a\in \{-s \ | -s\in \mathbb{N} \text{ and } \  -r+2\leq s <0\}
\end{eqnarray*}

and 
\begin{eqnarray*}
b=\frac{c_2}{2}-\frac{r^2}{2}+ra+\frac{a^2}{2}+\frac{3}{2}a.
\end{eqnarray*}  
\end{Theorem}

\begin{proof} 
As above,  by Remark \ref{Rem-zero} and Remark \ref{Remnotcritreg}, the proof of the theorem is restricted to coherent systems $(E,V) \in \widetilde{\mathbf{M}}_{\alpha}(2,2r,c_2,2)$ where $S(E)=2s \neq 0$. First we notice that, there is not a stable vector bundle $E$ such that $c_1(E)=2r$ and $c_2(E)=r^2, r^2+1$. This is because, if $c_1=2r$ and $c_2=r^2$, then Bogomolov inequality $\Delta=4c_2-c^2_1=0$. On the other hand, if $c_1=2r$ and $c_2=r^2+1$ and
suppose that such vector bundle $E$ exists, then we have that $c_1(E(-r))=0,$ $c_2(E(-r))=1$ and $S(E)=S(E(-r))=2m>0.$  Thus, according to Theorem \ref{invariantedesegre}, 
the vector bundle $E(-r)$ fits into the following exact sequence
\begin{eqnarray*}
0 \to \mathcal{O}_{\mathbb{P}^2}(-m) \to E(-r) \to \mathcal{O}_{\mathbb{P}^2}(m)\otimes I_Z \to 0
\end{eqnarray*}
where $Z$ is not contained in any  curve of degree $2m-1$ and $\ell(Z)=m^2+1.$ Since $Z$ is not contained in a curve of degree $2m-1$, it follows that $\ell(Z)\geq h^0(\mathbb{P}^2,\mathcal{O}_{\mathbb{P}^2}(2m-1))=2m^2+m$ which is impossible because $m\geq 1$.

Therefore, we are going to consider coherent systems
$(E,V)$ with $S(E)=2s$  and $s<0$. Using the same method from Theorem \ref{criticalThm1} (2), we obtain 
\begin{eqnarray*}
a=-s \mbox{ with } s\geq -r+2
\end{eqnarray*}
and 
\begin{eqnarray*}
b=\frac{c_2}{2}-\frac{r^2}{2}+ra+\frac{a^2}{2}+\frac{3}{2}a.
\end{eqnarray*}
\end{proof}

\begin{Theorem}\label{criticalThm3}
   Let $\widetilde{\mathbf{M}}_{\alpha}(2,c_1,c_2,2)$  be the moduli space of
$\alpha$-semistable coherent systems, $0<\alpha(m)=am+b\in \mathbb{Q}[m]$ and $c_1=2r$. If $r^2-s_{0}^{2}\leq c_2\le r^2-(s_0 + 1)^2$ for some integer $s_0$ satisfying $-r+2\leq s_0<0$, then $\alpha$ is a critical value if and only if $a$ is an integer number such that $-s_0\le a\le r-2$ and
       
\begin{eqnarray*}
b&=&\frac{1}{2}c_2-\frac{1}{2}r^2+ra +\frac{a^2}{2}+\frac{3}{2}a.
\end{eqnarray*}
\end{Theorem}

\begin{proof}
Assume that  $r^2-(s_0+1)^2>c_2\ge r^2-s_0^2,$
where $s_0$ is an integer such that $-r+2\leq s_0<0$.
Since $c_2<r^2$, Bogomolov
inequality implies that, if $(E,V)\in \widetilde{\mathbf{M}}_{\alpha}(2,c_1,c_2,2)$ then $E$ is an unstable vector bundle and hence $S(E)<0$.
Therefore, we need to consider coherent systems $(E,V)$ with $S(E)=2s<0.$
Now, under this hypothesis, we consider two situations: $\dim W_{max}\geq 1$ and $\dim W_{max}=0.$

\begin{enumerate}
\item Suppose that $\dim W_{max}\geq 1,$ then
\begin{eqnarray*}
\mu_a(E,V)-\mu_a(L_{max},W_{max})= s+a(1-\dim\,W_{max})\leq s<0
\end{eqnarray*}
and $(E,V)$ does not satisfy Corollary \ref{equivcritical}, i.e., does not determine a critical point.

\item Now we suppose that $\dim\, W_{max}=0.$ Then we have that 
\begin{eqnarray*}
\mu_a(E,V)-\mu_a(L_{max},W_{max})= s+a=0
\end{eqnarray*}
hence $a=-s$. Next, we compute all the possible values of $s$ based on the theorem's assumptions.

By hypothesis $r^2-(s_0+1)^2>c_2\ge r^2-s_0^2$ and $s<0,$ hence considering $j\in \mathbb{Z}$ we have the following statements
\begin{enumerate}
\item If $j\geq 1$, then $r^2-(s_0+j)^2\geq r^2-(s_0+1)^2>c_2.$
\item If $j\leq 0$, then $c_2\geq r^2-s_0^2\geq r^2-(s_0+j)^2.$
\end{enumerate}
Now by Theorem \ref{invariantedesegrenegativo}(1) there exists a vector bundle $E_j$ which fits into the following exact sequence
\begin{eqnarray*}
    0\to \mathcal{O}_{\mathbb{P}^2}(r-(s_0+j))\to E_j \to \mathcal{O}_{\mathbb{P}^2}(r+s_0+j)\otimes I_Z\to 0
    \end{eqnarray*}
    if and only if $j\leq 0.$ 
Moreover, $\mathcal{O}_{\mathbb{P}^2}(r-(s_0+j))\subset E_j$ is the maximal subbundle and $S(E_j)=-2(s_0+j).$
Therefore $a= -s_0-j \geq -s_0.$

 From Theorem \ref{chernpositiva} we have $s\geq -r+1.$ Therefore, we have that
\begin{eqnarray*}
-s_0 \le a\leq r-1.
\end{eqnarray*}
Now, if $s=-r+1,$ therefore $E$ fits the following exact sequence 
 \begin{eqnarray*}
0\to \mathcal{O}_{\mathbb{P}^2}(2r-1)\to E\to \mathcal{O}_{\mathbb{P}^2}(1)\otimes I_Z\to 0,
 \end{eqnarray*}
 therefore $\dim W_{max}=0$ if and only if $\ell(Z)=0,1$ or equivalent $c_2(E)=2r,2r-1.$
 But if $-r+2\leq s_0<0,$ it follows that  $c_2\geq 4r-4,$ and  $c_2(E)>2r.$ 
 Therefore $a=-s$ with $s\geq -r+2.$

Now, we see the condition for $b$: the condition that 
$ p^{\alpha}(E,V)=p^{\alpha}(L_{max},W_{max})$ is equivalent to
\begin{equation*}
c^{2}_1-2c_2-2c^{2}_1(L_{max})=(4b-6a)\left(\dim\,W_{max}- 1\right)
\end{equation*}
since $a=-s$ and $\dim W_{max}=0,$ the equation (\ref{bforpositives}) 
is equivalent to
\begin{eqnarray*}
b=\frac{1}{2}c_2-\frac{1}{2}r^2+ra +\frac{a^2}{2}+\frac{3}{2}a.
\end{eqnarray*}

\end{enumerate}
\end{proof}

Similarly, for the odd $c_1(E)$ case, we have the following result.

\begin{Theorem}\label{criticalThmodd}
Let $\widetilde{\mathbf{M}}_{\alpha}(2,c_1
,c_2,2)$  be the moduli space of
$\alpha$-semistable coherent systems with $0<\alpha(m)=am+b\in \mathbb{Q}[m]$,
and $c_1=2r-1$, then we have the following statements.
\begin{enumerate}


\item If $c_2\ge r^2+1,$ then $\alpha$ is a critical value if and only if satisfies one of the following:\\
\begin{enumerate}
    \item[(i)] $a\in \left\{s-\frac{1}{2}\ \vert \ s \in\mathbb{N}\ \text{and } \ 0<s<r \right\}$
and
\begin{eqnarray*}
b=\frac{r^2-a^2+2ra+2a-c_2-r}{2}+\frac{1}{8}.
\end{eqnarray*}

\item[(ii)] $ a=\frac{1}{2}-s, \text{ with } 0\geq s\geq -r+2$    
and      
\begin{eqnarray*}
b=\frac{a^2-r^2+2ra+2a+c_2+r}{2}-\frac{1}{8.}\\
\end{eqnarray*}
\end{enumerate}

\item If $r^2-r\leq c_2 <r^2+1,$ then $\alpha$ is a critical value if and only if  

\begin{eqnarray*}
a=\frac{1}{2}-s, \text{ with } 0\geq s\geq -r+3    
\end{eqnarray*}
and      
\begin{eqnarray*}
b=\frac{a^2-r^2+2ra+2a+c_2+r}{2}-\frac{1}{8}.\\
\end{eqnarray*}

       

   \item If $r^2-s_{0}^{2}-r+s_0\leq c_2< r^2-(s_0 + 1)^2+s_0-r+1$ with $r\geq 3$ and for some integer $s_0$ satisfying $-r+3\leq s_0<0$, then $\alpha$ is a critical value if and only if $a=a_{0}+\frac{1}{2}$ for any $a_{0}\in \mathbb{Z}_{+}$ such that $-s_0+\frac{1}{2}\le a< r-\frac{5}{2}$ 
   and
       
  \begin{eqnarray*}
      b=\frac{a^2-r^2+2ra+2a+c_2+r}{2}-\frac{1}{8}.
      \end{eqnarray*}
      \end{enumerate}
      \end{Theorem}
      
\begin{proof}
The proofs of $(1)$, $(2)$, and $(3)$ are similar to those of Theorems \ref{criticalThm1}, \ref{criticalThm2}, and \ref{criticalThm3}, respectively.

\end{proof}

\section{Flips}\label{sect-Flips}

In order to study the differences between consecutive moduli spaces, 
we follow \cite[Definition 6.4]{Bradlow}. 
Hence, if $\alpha_{i,t},\alpha_{i+1,t}$ are two consecutive critical values,
then for any $\alpha, \alpha'$ in $(\alpha_{i,t},\alpha_{i+1,t})$ it follows that $\mathbf{M}_{\alpha}(2,c_1,c_2, 2)= \mathbf{M}_{\alpha'}(2,c_1,c_2, 2)$. In this case, we denote the moduli by
$\mathbf{M}_{i}(2,c_1,c_2, 2)$. 

Clearly, the critical values depend on the type of coherent system. In the previous section, we determine the different critical values when $c_1$ is odd or even, and for this, we introduced the parameter $t\in \{0,1\}$. Thus, for coherent systems of type $(2,2r-t,c_2,2)$  we denote the critical values as $\alpha_{i,t}.$
Now, if $\mathbf{M}_{i}(2,2r-t,c_2, 2)$ denotes the moduli space in the interval $(\alpha_{i,t},\alpha_{i+1,t})$ we define the following sets:

\[\Sigma^{+}_{i,t}:=  \{ (E,V)\in  \mathbf{M}_{i}(2,2r-t,c_2,2)  ~\vert~ \text{$(E,V)$ is not $\alpha$-stable for $\alpha<\alpha_i$}\} 
\]

and 
\[
\Sigma^{-}_{i-1,t} = \{ (E,V)\in  \mathbf{M}_{i-1}(2,2r-t,c_2,2)  ~\vert~ \text{$(E,V)$ is not $\alpha$-stable for $\alpha>\alpha_{i}$} \}. 
\]

The purpose of this section is to describe the topological
and geometric properties of $\Sigma^{+}_{i,t}$ and
$\Sigma^{-}_{i-1,t}$. For this, we have defined them using the
Segre invariant depending on the critical value.  
Hence, if $\alpha_{i,t}=a_im+b_i$ is a critical value that
satisfies some of the conditions of Theorem \ref{criticalThm1}(1) (see also Theorem \ref{criticalThmodd} (1(i)). 
Then we have the following;

\[
\Sigma^{-}_{i-1,t}:=\{(E,V)\in \mathbf{M}_{i-1,}(2,2r-t,c_2,2) ~\vert~ (L_{max},W_{max})=(\mathcal{O}_{\mathbb{P}^2}(r-a_i), W) \text{ and } \dim W=2\}
\]
and  $\Sigma^{+}_{i,t}=\emptyset$. \\

Notice that if   $(E,V)\in \Sigma^{-}_{i-1,t}$  then $(L_{max},W_{max})=(\mathcal{O}_{\mathbb{P}^2}(r-a_i), V).$
Moreover, $(E,V)\in \Sigma^{-}_{i-1,t}$ if and only if there exists the following
exact sequence 

\begin{eqnarray}\label{sucnegativot}
0\to (\mathcal{O}_{\mathbb{P}^2}(r-s),V)\to (E,V)\to (\mathcal{O}_{\mathbb{P}^2}(r+s-t)\otimes I_{Z_t},0)\to 0,
\end{eqnarray}

 \noindent where $Z_t\subset \mathbb{P}^2$ is of codimension 2,  $\ell(Z_t)= c_2+s^2-r^2+t(r-s)$ and it is not contained
in a curve of degree $2s-1-t.$   
Notice that from the definition, if $(E,V)\in  \Sigma^{-}_{i-1,t}$ then $S(E)=2a_i-t.$\\






In the following, we study the geometric properties of $\Sigma^{-}_{i-1,t}$ through extensions of type (\ref{sucnegativot}). The following result describes some properties of these extensions.\\

    \begin{Lemma}\label{cohomologiaext} Let $s,r,c_2$ be  natural numbers, $t\in \{0,1\}$ and $Z_t\subset \mathbb{P}^2$ of codimension $2$ not contained in a curve of degree $2s-1-t$ with $\ell(Z_t)=c_2+s^2-r^2+t(r-s)$. Then
\begin{enumerate}
\item $Hom((\mathcal{O}_{\mathbb{P}^2}(r+s-t)\otimes I_{Z_t},0),(\mathcal{O}_{\mathbb{P}^2}(r-s),V))=0$.
\item $\text{Ext}^2((\mathcal{O}_{\mathbb{P}^2}(r+s-t)\otimes I_{Z_t},0),(\mathcal{O}_{\mathbb{P}^2}(r-s),V))=0.$ 
\item $\dim \text{Ext}^1((\mathcal{O}_{\mathbb{P}^2}(r+s-t)\otimes I_{Z_t},0),(\mathcal{O}_{\mathbb{P}^2}(r-s)),V))=c_2-r^2-s^2+3s-1 +\frac{t}{2}(2r+2s-3-t)$.
\end{enumerate}
\end{Lemma}

\begin{proof}
Let $\Gamma=(\mathcal{O}_{\mathbb{P}^2}(r+s-t)\otimes I_{Z_t},0)$ be a coherent system, where $Z_t\subset \mathbb{P}^2$ is of codimension 2 and is not contained in a curve of degree $2s-1-t$. Let $\Gamma^{'}=(\mathcal{O}_{\mathbb{P}^2}(r-s),V)$ be a coherent system with $\dim V=2.$  By \cite[Corollary $2.6$]{Min} we obtain the following exact sequence

\begin{equation}\label{exactMin}
\begin{split} 
0 &\to \text{Hom}(\Gamma,\Gamma^{'}) 
\to \text{Hom}(\mathcal{O}_{\mathbb{P}^2}(r+s-t)\otimes I_{Z_t},\mathcal{O}_{\mathbb{P}^2}(r-s)) 
\to \text{Hom}(0,H^0(\mathbb{P}^2,\mathcal{O}_{\mathbb{P}^2}(r-s)/V)) \\
&\to \text{Ext}^1(\Gamma,\Gamma^{'}) \to \text{Ext}^1(\mathcal{O}_{\mathbb{P}^2}(r+s-t)\otimes I_{Z_t},\mathcal{O}_{\mathbb{P}^2}(r-s)) \to \text{Hom}(0,H^1(\mathbb{P}^2,\mathcal{O}_{\mathbb{P}^2}(r-s)) \to \\
 &\to  \text{Ext}^2(\Gamma,\Gamma^{'}) \to \text{Ext}^2(\mathcal{O}_{\mathbb{P}^2}(r+s-t)\otimes I_{Z_t},\mathcal{O}_{\mathbb{P}^2}(r-s)) \to  \cdots
\end{split}
\end{equation}

It follows that;

\begin{enumerate}
    \item Since $\text{Hom}(\mathcal{O}_{\mathbb{P}^2}(r+s-t)\otimes I_{Z_t},\mathcal{O}_{\mathbb{P}^2}(r-s))=0$, we have
    \begin{eqnarray*}
 \text{Hom}(\Gamma,\Gamma^{'})= \text{Hom}((\mathcal{O}_{\mathbb{P}^2}(r+s-t)\otimes I_{Z_t},0),(\mathcal{O}_{\mathbb{P}^2}(r-s),V))=0.
    \end{eqnarray*}
This proves the first statement.

    \item Now observe that,  \begin{eqnarray*}
\text{Ext}^{2}(\mathcal{O}_{\mathbb{P}^2}(r+s-t)\otimes I_{Z_t}, \mathcal{O}_{\mathbb{P}^2}(r-s))
&\cong& \text{Ext}^{0}(\mathcal{O}_{\mathbb{P}^2}(r-s), \mathcal{O}_{\mathbb{P}^2}(r+s-t)\otimes I_{Z_t}\otimes \omega_{\mathbb{P}^2})\\
&\cong& \text{Ext}^{0}(\mathcal{O}_{\mathbb{P}^2},\mathcal{O}_{\mathbb{P}^2}(s-r)\otimes \mathcal{O}_{\mathbb{P}^2}(r+s-t)\otimes I_{Z_t}\otimes \omega_{\mathbb{P}^2})\\
&\cong& \text{Ext}^{0}(\mathcal{O}_{\mathbb{P}^2},\mathcal{O}_{\mathbb{P}^2}(2s-3-t)\otimes I_{Z_t})\\
&\cong& H^{0}(\mathbb{P}^2,\mathcal{O}_{\mathbb{P}^2}(2s-3-t)\otimes I_{Z_t}).
\end{eqnarray*}
Since $Z_t$ is not contained in a curve of degree $2s-1-t$, 
it follows that $h^0(\mathbb{P}^2,\mathcal{O}_{\mathbb{P}^2}(2s-3-t)\otimes I_{Z_t})=0.$ 
Hence $\text{Ext}^2(\Gamma,\Gamma^{'})=0$ which proves the second statement.

\item From the exact sequence (\ref{exactMin}), we have $\text{Ext}^1(\Gamma,\Gamma^{'}) \cong \text{Ext}^1(\mathcal{O}_{\mathbb{P}^2}(r+s-t)\otimes I_{Z_t},\mathcal{O}_{\mathbb{P}^2}(r-s))$. Moreover,
\begin{eqnarray*}
\text{Ext}^1(\mathcal{O}_{\mathbb{P}^2}(r+s-t)\otimes I_{Z_t},\mathcal{O}_{\mathbb{P}^2}(r-s)) &\cong &\text{Ext}^1(\mathcal{O}_{\mathbb{P}^2}(r-s),\mathcal{O}_{\mathbb{P}^2}(r+s-t-3)\otimes I_{Z_t})\\
& \cong & \text{Ext}^1(\mathcal{O}_{\mathbb{P}^2},\mathcal{O}_{\mathbb{P}^2}(2s-3-t)\otimes I_{Z_t})\\
&\cong & H^1(\mathcal{O}_{\mathbb{P}^2}(2s-3-t)\otimes I_{Z_t}).
\end{eqnarray*}
 Now, from the exact sequence, 
\begin{eqnarray*}
0\to \mathcal{O}_{\mathbb{P}^2}(2s-3-t)\otimes I_{Z_t} \to \mathcal{O}_{\mathbb{P}^2}(2s-3-t) \to \mathcal{O}_{Z_t}\to 0
\end{eqnarray*}
and since that  $Z_t$ is not contained in a curve of degree $2s-1-t$, it follows that $h^0(\mathbb{P}^2,\mathcal{O}_{\mathbb{P}^2}(2s-3-t)\otimes I_{Z_t})=0$ and
\begin{eqnarray*}
h^1(\mathcal{O}_{\mathbb{P}^2}(2s-3)\otimes I_{Z_t})=\ell(Z_t)-h^0(\mathbb{P}^2,\mathcal{O}_{\mathbb{P}^2}(2s-3-t)).
\end{eqnarray*}
Therefore, 
\begin{eqnarray*}
\dim \text{Ext}^1(\mathcal{O}_{\mathbb{P}^2}(r+s-t)\otimes I_{Z_t},\mathcal{O}_{\mathbb{P}^2}(r-s))&=&\ell(Z_t)-s(2s-3)-1+\frac{t}{2}(4s-3-t)\\
&=&c_2-r^2-s^2+3s-1 +\frac{t}{2}(2r+2s-3-t).
\end{eqnarray*}

\end{enumerate}

\end{proof}

Now we will consider the universal extension that parameterizes
extensions of type $(\ref{sucnegativot}).$ 
Following \cite{Lange} (see also \cite{BGMMN}) we construct such an extension as follows.

Notice that coherent systems $(\mathcal{O}_{\mathbb{P}^2}(r-s),V)$ with $\dim V=2$ 
are parameterized by the Grassmannian variety $G: =Gr(2, H^{0}(\mathbb{P}^2,\mathcal{O}_{\mathbb{P}^2}(r-s)).$ 
It follows that we have a family of coherent systems 
$(\mathcal{E}_1, \mathcal{V}_1):=(p_{1}^{*}\mathcal{O}_{\mathbb{P}^2}(r-s),p_{2}^{*}S)$ 
over $\mathbb{P}^{2}\times G$ where $S$ denotes the tautological bundle over $G$ of rank two defined by
the tautological exact sequence
$$
0\rightarrow S \rightarrow H^{0}(\mathbb{P}^2,\mathcal{O}_{\mathbb{P}^2}(r-s))\otimes \mathcal{O}_G\rightarrow Q\rightarrow 0.
$$

On the other hand, for the parameter $t\in \{0,1\},$  the coherent systems $(\mathcal{O}_{\mathbb{P}^2}(r+s-t)\otimes I_{Z_t},0)$ such that $Z_t\subset \mathbb{P}^2$ is not contained in a curve of degree $2s-1-t$ are parameterized by an open set $H_t \subset \text{Hilb}^{\ell(Z_t)}(\mathbb{P}^{2})$. That is,
\begin{eqnarray*}
H_t:=\{Z_t\in \text{Hilb}^{\ell(Z_t)}(\mathbb{P}^2) \mid h^0(\mathbb{P}^2,\mathcal{O}_{\mathbb{P}^2}(r+s-t)\otimes I_{Z_t})=0  \}.
\end{eqnarray*}

This implies that we have a family of coherent systems over
$\mathbb{P}^2\times H_t$ given by 
$(\mathcal{E}_2, \mathcal{V}_2):=(p_{1}^{*}\mathcal{O}_{\mathbb{P}^2}(r+s-t)\otimes \mathcal{I}_{H_t},0)$, where 
$\mathcal{I}_{H_t}$ is the ideal sheaf of the universal scheme $\mathcal{U}|_{H_t}$ on $\mathbb{P}^2\times H_t.$\\








\begin{Theorem}\label{ThmFlip}
There exists a vector bundle $W^{-}_t$ over $G\times H_t$  
such that $\Sigma^{-}_{i,t}=\mathbb{P}(W^{-}_t)$. 
In particular, $\Sigma^{-}_{i,t}$  is irreducible of dimension 
\begin{eqnarray*}
3c_2-3r^2+s^2+3s+2h^0(\mathbb{P}^2,\mathcal{O}_{\mathbb{P}^2}(r-s))-6+\frac{t}{2}(6r-2s-3-t).
\end{eqnarray*}
\end{Theorem}
\begin{proof}
 The proof is similar to \cite[A.10]{BGMMN}. By Lemma \ref{cohomologiaext}, we get
\begin{eqnarray*}
\text{dim } \text{Ext}^1 ((\mathcal{O}_{\mathbb{P}^2}(r+s-t)\otimes I_{Z_t},0),(\mathcal{O}_{\mathbb{P}^2}(r-s), V))
\end{eqnarray*}
is independent of the choice of $(\mathcal{O}_{\mathbb{P}^2}(r+s-t)\otimes I_{Z_t},0)\in H_t$, $(\mathcal{O}_{\mathbb{P}^2}(r-s), V)\in G$. It follows from \cite[Corollaire 1.20]{Min} that there is a vector bundle $W^-_t$ on $H_t\times G$ whose fiber over
$ ((\mathcal{O}_{\mathbb{P}^2}(r+s-t)\otimes I_{Z_t},0),(\mathcal{O}_{\mathbb{P}^2}(r-s), V))\in  H_t \times  G$ is
\begin{eqnarray*}
 \text{Ext}^1((\mathcal{O}_{\mathbb{P}^2}(r+s-t)\otimes I_Z,0),(\mathcal{O}_{\mathbb{P}^2}(r-s), V)),
\end{eqnarray*}
where
 \begin{eqnarray*}
W^{-}_t = \mathcal{E}xt^1_{\pi}((p_2\times Id)^{*}(\mathcal{O}_{\mathbb{P}^2}(r+s-t)\otimes I_{Z_t},0), (p_1 \times Id)^{*}(\mathcal{O}_{\mathbb{P}^2}(r-s), V)),
\end{eqnarray*}
with $\pi : H_t \times G \times \mathbb{P}^2 \to H_t \times G$, $p_1 : H_t\times G \to  H_t$ and $p_2 : H_t\times G \to G$ are the natural projections.

Hence $\mathbb{P}(W^{-}_t)$  classifies the non-trivial 
extensions (\ref{sucnegativot}) up to scalar multiples. 
Therefore, we can define 
$f: \mathbb{P}(W^{-}_t)\rightarrow \mathbf{M}_{i}(2,2r-t,c_2,2)$ 
set theoretically as the natural map that sends (\ref{sucnegativot}) to
$ (E,V) \in  \mathbf{M}_{i}(2,2r-t,c_2,2)$, 
and the image of $f$ is exactly $\Sigma^{-}_{i,t}$.

To prove that $f$ is a morphism, we will construct
an universal extension over $\mathbb{P}(W^-_t) \times \mathbb{P}^2$.
Let $\sigma:\mathbb{P}(W^-_t)\times \mathbb{P}^2 \to  \mathbb{P}(W^-_t)$
and $p:\mathbb{P}(W^-_t)\to H_t \times G$ the natural projections.
We write for $ m = 1, 2,$ 
\begin{eqnarray*}
(\mathcal{E}_m, \mathcal{V}_m)^-_t = (p \times Id)^{*}(p_m \times Id)^{*}(\mathcal{E}_m, \mathcal{V}_m).
\end{eqnarray*}
We construct the universal extension as 
\begin{eqnarray}\label{extensionuniversalt}
0 \to (\mathcal{E}_1, \mathcal{V}_1)^-_t \otimes \sigma^{*} \mathcal{O}_{\mathbb{P}W^-_t}(1)\to  (\mathcal{E}, \mathcal{V})^-_t\to  (\mathcal{E}_2, \mathcal{V}_2)^-_t\to 0
\end{eqnarray}
on $\mathbb{P}(W^-_t) \times \mathbb{P}^2$. Extensions of the form (\ref{extensionuniversalt}) are classified by
\begin{eqnarray*}
\text{Ext}^1((\mathcal{E}_2, \mathcal{V}_2)^-_t),(\mathcal{E}_1, \mathcal{V}_1)^-_t \otimes  \sigma^{*}\mathcal{O}_{\mathbb{P}W^-_t}(1)).
\end{eqnarray*}
We have a spectral sequence whose $E_2$-term is given by
\begin{eqnarray*}
E^{pq}_2=H^p(\mathcal{E}xt^q_{\sigma}((\mathcal{E}_2, \mathcal{V}_2)^-_t),(\mathcal{E}_1, \mathcal{V}_1)^-_t\otimes  \sigma^{*}\mathcal{O}_{\mathbb{P}W^-_t}(1))).
\end{eqnarray*}
We obtain the following exact spectral sequence.
\begin{align*}
 H^1(\mathcal{H}om_{\sigma}((\mathcal{E}_2,\mathcal{V}_2)^-_t),(\mathcal{E}_1, \mathcal{V}_1)^-_t \otimes  \sigma^{*}\mathcal{O}_{\mathbb{P}W^-_t}(1))) \to \text{Ext}^1(((\mathcal{E}_2, \mathcal{V}_2)^-),(\mathcal{E}_1, \mathcal{V}_1)^- \otimes  \sigma^{*}\mathcal{O}_{\mathbb{P}W^-_t}(1)))\to \\
 H^0(\mathcal{E}xt^1_{\sigma}((\mathcal{E}_2,\mathcal{V}_2)^-_t),(\mathcal{E}_1, \mathcal{V}_1)^- \otimes  \sigma^{*}\mathcal{O}_{\mathbb{P}W^-_t}(1)))\to H^2(\mathcal{H}om_{\sigma}((\mathcal{E}_2, \mathcal{V}_2)^-_t),(\mathcal{E}_1, \mathcal{V}_1)^-_t \otimes  \sigma^{*}\mathcal{O}_{\mathbb{P}W^-_t}(1))).
\end{align*}
By Lemma \ref{cohomologiaext}, it follows that $\mathcal{H}om_{\sigma}((\mathcal{E}_2,\mathcal{V}_2)^-_t),(\mathcal{E}_1, \mathcal{V}_1)^-_t \otimes  \sigma^{*}\mathcal{O}_{\mathbb{P}W^-_t}(1))=0$, which implies the isomorphism
\begin{eqnarray*}
\text{Ext}^1(((\mathcal{E}_2, \mathcal{V}_2)^-),(\mathcal{E}_1, \mathcal{V}_1)^-_t \otimes  \sigma^{*}\mathcal{O}_{\mathbb{P}W^-_t}(1)))\cong H^0(\mathcal{E}xt^1_{\sigma}((\mathcal{E}_2,\mathcal{V}_2)^-_t),(\mathcal{E}_1, \mathcal{V}_1)^-_t\otimes  \sigma^{*}\mathcal{O}_{\mathbb{P}W^-_t}(1))).
\end{eqnarray*}
On the other hand, by base-change \cite[Theoreme 1.16]{Min},
\begin{align*}
\mathcal{E}xt^1_{\sigma}((\mathcal{E}_2, \mathcal{V}_2)^-_t),&(\mathcal{E}_1, \mathcal{V}_1)^-_t\otimes  \sigma^{*}\mathcal{O}_{\mathbb{P}W^-_t}(1))= \mathcal{E}xt^1_{\sigma}((\mathcal{E}_2, \mathcal{V}_2)^-_t),(\mathcal{E}_1, \mathcal{V}_1)^-_t)\otimes \mathcal{O}_{\mathbb{P}W^-_t}(1))\\
&= p^{*}\mathcal{E}xt^1_{\pi}((p_2 \times Id)^{*}((\mathcal{E}_2, \mathcal{V}_2)^-_t),(p_1 \times  Id)^{*}(\mathcal{E}_1, \mathcal{V}_1)^-_t) \otimes \mathcal{O}_{\mathbb{P}W^-_t}(1)\\
&= p^{*}W^-_t \otimes \mathcal{O}_{\mathbb{P}W^-_t}(1).
\end{align*}

Now, $H^0(p^{*}W^-_t \otimes \mathcal{O}_{\mathbb{P}W^-_t}(1))=\text{End}(W^{-}_t)$.  The universal extension (\ref{extensionuniversalt}) corresponds to the
identity endomorphism of $W^-_t$. And the restriction
of (\ref{extensionuniversalt}) to $  \{y\} \times \mathbb{P}^2$ is precisely the extension
\begin{eqnarray*}
0\to (\mathcal{E}_{1},\mathcal{V}_1)^{-}_t
\otimes \sigma^{*}\mathcal{O}_{\mathbb{P}(F)}(1)\to (\mathcal{E},\mathcal{V})^{-}_t\to (\mathcal{E}_2,\mathcal{V}_2))^{-}_t\to 0
\end{eqnarray*}

\noindent that corresponds to $y \in \mathbb{P}(W^{-}_t)$. Therefore, the morphism $\mathbb{P}(W^-_t)\to  \mathbf{M}_{i}(2,2r-t,c_2,2)$  
given by (\ref{extensionuniversalt}) and the universal property of $ \mathbf{M}_{i}(2,2r-t,c_2,2)$
coincide with $f$.

Since $H_t\times G$ is smooth and irreducible of dimension 
\begin{eqnarray*}
2\ell(Z_t)+2(h^0(\mathbb{P}^2,\mathcal{O}_{\mathbb{P}^2}(r-s))-2),
\end{eqnarray*}
and $W^{-}_t$ is a vector bundle over $H_t\times G$, it follows that $\Sigma^-_{i,t}:=\mathbb{P}(W^-_t)$ is irreducible of dimension
\begin{eqnarray*}
&&2\ell(Z_t)+2(h^0(\mathbb{P}^2,\mathcal{O}_{\mathbb{P}^2}(r-s))-2)+\dim \text{Ext}^1((\mathcal{O}_{\mathbb{P}^2}(r+s-t)\otimes I_{Z_t},0),(\mathcal{O}_{\mathbb{P}^2}(r-s)),V)) -1\\
&=&
3c_2-3r^2+s^2+3s+2h^0(\mathbb{P}^2,\mathcal{O}_{\mathbb{P}^2}(r-s))-6+\frac{t}{2}(6r-2s-3-t).
    \end{eqnarray*}
\end{proof}

\end{document}